\documentclass[a4paper, 12pt]{article}
    
\usepackage{caption}
\usepackage{subcaption}

\usepackage[usenames,dvipsnames]{color}

\usepackage{amsmath}

\usepackage{graphicx}

\usepackage{amsthm}

\usepackage{amssymb}

\usepackage[mathcal]{euscript}

\usepackage{hyperref}

\usepackage{enumerate}

\definecolor{viola}{rgb}{0.3,0,0.7}
\definecolor{ciclamino}{rgb}{0.5,0,0.5}
\definecolor{rosso}{rgb}{0.8,0,0}

\newtheorem{theorem}{Theorem}[section]
\newtheorem{remark}[theorem]{Remark}
\newtheorem{corollary}[theorem]{Corollary}
\newtheorem{definition}[theorem]{Definition}
\newtheorem{proposition}[theorem]{Proposition}
\newtheorem{lemma}[theorem]{Lemma}
\newtheorem{example}[theorem]{Example}

\def\erre{{\mathbb{R}}}

\let\theta\vartheta

\let\phi\varphi

\def\w2{\stackrel{2}{\rightharpoonup}}

\def\strong2{\stackrel{2}{\rightarrow}}

\definecolor{darkred}{RGB}{153,0,0}
\definecolor{darkblue}{RGB}{0,0,204}

\newcommand{\R}{\mathbb{R}}

\def\RR{\mathbb{R}}

\def\M{\mathfrak{M}}
\def\h{\mathfrak h}
\def\H{\mathcal H}
\def\R{\mathcal R}
\def\L{\mathcal L}
\def\V{\mathcal V}

\def\tr{\mathrm{Tr}}

\def\bra#1{\langle #1|}
\def\ket#1{|#1\rangle}

\def\ketbra#1#2{|#1\rangle\langle#2|}

\def\id{\mathrm{Id}}
\def\TT{\mathcal{T}}
\def\AV{A({\cal V})}

\def\nn{\mathbb{N}}
\def\zz{\mathbb{Z}}
\def\rr{\mathbb{R}}
\def\cc{\mathbb{C}}

\def\M{\mathfrak{M}}
\def\L{\mathfrak L}
\def\HH{{\mathcal H}}
\def\hh{{\mathfrak h}}

\def\RR{{\mathcal R}}
\def\pp{{\mathbb P}}

\title{On a generalized Central Limit Theorem\\
and Large Deviations\\
for Homogeneous Open Quantum Walks}

\author{Raffaella Carbone\textsuperscript{$*$}, Federico Girotti\textsuperscript{$\dagger$} and Anderson Melchor Hernandez\textsuperscript{$\star$}\\Dipartimento di Matematica dell'Universit\`a di Pavia\\ via Ferrata 1, 27100 Pavia, Italy\\\textsuperscript{$*$}e-mail: \texttt{raffaella.carbone@unipv.it}\\ \textsuperscript{$*$}ORCID: \texttt{https://orcid.org/0000-0002-3005-4216}\\\textsuperscript{$\dagger$}e-mail: \texttt{f.girotti@campus.unimib.it }\\
\textsuperscript{$\dagger$}ORCID: \texttt{https://orcid.org/0000-0001-9303-1428}
\\\textsuperscript{$\star$}e-mail: \texttt{anderson.melchorhernand01@universitadipavia.it } }

\begin{document}

\maketitle

\begin{abstract}
\parbox{0.85\textwidth}{ 
We consider homogeneous open quantum random walks on a lattice with finite dimensional local Hilbert space and we study in particular the position process of the quantum trajectories of the walk. 
We prove that the properly rescaled position process asymptotically approaches
 a mixture of Gaussian measures.
 We can generalize the existing central limit type results and give more explicit expressions for the involved asymptotic quantities, dropping any additional condition on the walk.
 We use deformation and spectral techniques, together with reducibility properties of the local channel  associated with the open quantum walk. Further, we can provide a large deviations' principle in the case of a fast recurrent local channel and at least lower and upper bounds in the general case.

} 
\end{abstract}


\setlength{\parindent}{0pt}

\section{Introduction}  \label{Int}

Quantum walks are interesting mathematical objects, introduced in discrete and continuous time, widely studied for their applications in various fields. They can be thought of as Markov processes on a lattice where the evolution of the walker's position depends on local degrees of freedom.

A natural need for a definition of quantum random walks in an open environment arises also because unitary quantum walks are often of difficult physical implementation due to decoherence effects.
We study open quantum random walks (OQRW) as introduced in \cite{APSS} by Attal et al.. Such processes are a possible noncommutative generalization of classical Markov Chains and have applications in quantum computing, quantum optics, biology. In the last years they have been intensively studied, we refer to \cite{PS2} for a recent survey on the subject and e.g. to \cite{APS,MSKPE,PS3} for some applications. 
The quantum trajectories associated with an OQRW defined on a lattice $V$ produce a classical $V$-valued process which is not Markov in general. It has been the object of different studies and particular attention has been concentrated from the origin on the case of homogeneous OQRWs (HOQRWs) on infinite lattices. In this case, the quantum process has no invariant state and an interesting question is investigating the conditions for different kinds of asymptotic behaviors of the position process, such as the law of large numbers, central limit theorems, and large deviations (\cite{AGS,CP,KoYo, KKSY,KonYo,Pas}).

In \cite{AGS}, a central limit theorem (CLT) and a law of large numbers for particular HOQRWs was proved by the use of the Poisson equation and the CLT theorem for martingales. At the same time, the authors highlighted the difficulty to prove analogous results under weaker assumptions, even if the possibility of appearance of a mixture of Gaussians was already clear (\cite{PS}).
In \cite{CP}, a large deviation principle was also proved, and a different approach was used to prove the CLT, but the assumptions essentially remained the same. 
Since then, various papers have been devoted to investigating this problem: see for instance \cite{BrH,KoYo, KKSY, KonYo,Pe,Pas} (see Remark \ref{comparison} for more details).

Before proceeding, let us now introduce more precisely the mathematical objects we treat.
\begin{itemize}
\item 
Let $V \subset \rr^d$ denote a locally finite lattice, which without loss of generality we assume contains $0$, and is positively generated by a set $S=\{s_1,\dots,s_v\} \neq \{0\}$. The canonical example is $V=\zz^d$ and $S=\{\pm e_1,\dots, \pm e_d\}$ where $(e_1,\dots,e_d)$ is the canonical basis of $\rr^d$. 
\item
We denote by $\ell^2(V)$ the Hilbert space of square summable sequences indexed in $V$, describing the position of the particle in the quantum evolution, and we introduce a finite-dimensional Hilbert space $\h$ describing the internal degrees of freedom of the particle. 
 We fix $\{ \ket{\underline{k}}\}_{\underline{k} \in V}$ an orthonormal basis for $\ell^2(V)$. 
 \item
We consider a quantum system described by the separable complex Hilbert space $\H=\h\otimes \ell^2(V)$.
We denote by $B(\H)$  the von Neumann algebra of bounded linear operators on $\H$ and by $L^1(\HH)$ trace class operators on $\H$ (similarly for $B(\h)$ and  $L^1(\h)$).
Self-adjoint bounded operators correspond to physical observables, while unit-trace positive operators are called states and represent the noncommutative analog of probability densities.
\end{itemize}

A HOQRW (\cite{APSS}) $\M$ is a particular quantum channel acting on $L^1(\H)$ in such a way that, at each time step, the position of the evolution can go only to nearest neighbors and also the change in the local state only depends on the position shift. More precisely, $\M$ is defined through its Kraus form as
\begin{equation}\label{oqw1}
\begin{split}
\M: L^1(\H) &\rightarrow L^1(\H) \\
\omega
&\mapsto \sum_{\underline{k} \in V}\sum_{i=1}^v 
(L_i \otimes \ket{\underline{k}+s_i}\bra{\underline{k}})
\,  \omega \,
(L_i^*\otimes \ket{\underline{k}}\bra{\underline{k}+s_i}),
\end{split}
\end{equation}
where 
$\{L_i\}_{i=1}^v$ are operators in $B(\h)$ such that
$\sum_{i=1}^vL_i^*L_i=1_\hh$.
The auxiliary (or local) map is the quantum channel $\L$ on the space $L^1(\h)$ defined by
\begin{equation*}
\L:L^1(\h) \rightarrow L^1(\h), \qquad
\L(\sigma) = \sum_{i=1}^v L_i \sigma L_i^{*}.
\end{equation*}
We shall see that this auxiliary map is of primary importance in our study:  it completely characterizes $\mathfrak{M}$ and so it contains all essential information.

Given the open quantum random walk $\M$, we can then fix an initial state $\rho$ (a positive unit-trace operator in $L^1(\HH)$), and, following the usual construction for quantum trajectories, we can introduce the process $(X_n,\rho_n)_{n\geq 0}$, keeping track of the position $X_n$, valued in $V$, and of the internal state $\rho_n$ of the particle (a positive unit-trace operator in $L^1(\h)$). See subsection \ref{section:QT} for more precise definitions.

As we already mentioned, the main topic of this paper is about central limit type results and large deviations for the position process $(X_n)_n$.
The existing results were established assuming different conditions about the local map $\L$. Even if the terminology is not always the same, we can say that precise CLT and large deviations' principle have been proved only under the assumptions that the local map has a unique invariant state (sometimes also faithful).
Some partial results were obtained under particular reducibility conditions on $\L$ in the case it is fast recurrent (e.g. \cite{AGS} and \cite{KoYo}).
Our aim is to establish the best results in these directions without any restriction on the local map.  
For central limit type results, we can determine the mixture of Gaussians which asymptotically approaches the law of the position process: the description is complete of exact parameters of the involved Gaussians, the coefficients in the mixture are expressed through quantum absorption probabilities, and we can explicitly use a distance on the space of probability laws.

For both large deviations and central limit theorem, we use deformation techniques and spectral theory, (and G\"artner-Ellis' and Bryc's theorems, respectively), following some ideas and lines already used for the irreducible and fast recurrent case in \cite{CP}.
In order to tackle the difficulties due to the more general context, our main additional tools will base on noncommutative probabilistic features of the local channel $\L$: the (in general non-unique) decomposition of the local space in irreducible invariant domains (or enclosures) and the quantum absorption properties of the same domains.

In Section \ref{sec:intro} we shall describe the construction of quantum trajectories associated with the HOQRW and recall some basic notions about invariant domains and absorption operators. Then, at the beginning of Section \ref{sec:mixgauss}, we shall go back to these topics and add more details about the reducibility properties of a quantum channel and the associated decomposition of the local space $\hh$ in invariant domains.

In Section \ref{sec:singlegauss}, we shall determine a family of probability measures under which the position process verifies a central limit theorem. 
These probability measures are absolutely continuous with respect to the standard measure $\pp_\rho$, induced by the initial state $\rho$ of the evolution, and 
are naturally associated with the invariant domains of the local channel. The densities of these measures and the parameters of the limit Gaussian are explicitly written in terms of the initial state and of the particular invariant domain.
 
Then, in Section \ref{sec:mixgauss}, we shall go to the general case using the decomposition of the space $\hh$ and deducing an expression of $\pp_\rho$ as convex combinations of probability measures described in the previous section.
Without giving all the details, we can write the statement of the main theorem of the section,  a ``generalized CLT'' (see Theorem \ref{thm:clt} for a more precise statement):

\begin{theorem}  
Let $\pp_{\rho,n}$ the law of the process $\frac{X_n-X_0}{\sqrt{n}}$ under the measure $\pp_\rho$ induced by the initial state $\rho$ of the quantum process. Then
\[\lim_{n \rightarrow +\infty} {\rm dist}\left (\pp_{\rho,n}, \,\sum_{\alpha \in A} a_{\alpha}(\rho) {\mathcal N}(\sqrt{n} m_\alpha,D_{\alpha}) \right )=0
\]
where ${\mathcal N}(\sqrt{n} m_\alpha,D_{\alpha})$ is a Gaussian measure with mean $\sqrt{n} m_\alpha$ and covariance matrix $D_{\alpha}$; the parameters $a_{\alpha}(\rho)$, $m_\alpha$, $D_{\alpha}$ are explicitly determined and depend on the initial state of the walk and on the proper decomposition of the local space 
$\hh= \underset{\alpha\in A} {\oplus} \, \chi_\alpha \oplus {\cal T} $ associated with
 the local channel $\L$.
\end{theorem}

For results on large deviations, the statement is more complicated and we directly refer the reader to Section \ref{sec:LDP}. Here, we can simply anticipate that we can prove a large deviation principle in the case the local channel is fast recurrent (i.e. there exists at least one faithful invariant state), while, when there is a non trivial transient subspace, we can only find upper and lower bounds through G\"artner-Ellis' theorem. In both cases we can  explicitly write the rate functions using the same ingredients as above.

Finally, we discuss some examples and numerical simulations in Section \ref{sec:examples}.

\section{Preliminaries and context description} \label{sec:intro}

We recall here some basic definitions, notations, and existing results. In the first subsection, we introduce the precise definition of quantum trajectories. In the second one, we recall some general notions about invariant domains (or enclosures) and absorption operators.

A quantum channel is for us a completely positive trace-preserving linear map acting on trace class operators. We already introduced the quantum channel $\M$, acting on $L^1(\H)$, and defining the HOQRW, and the local channel $\L$ acting on the trace class operators $L^1(\hh)$ on the local Hilbert space. 

Notice that the evolution of the system described by a HOQRW depends only on the matrix elements of the state which are diagonal with respect to the position observable, hence we can assume that the initial state $\rho$ of the system is of diagonal form
\[\rho=\sum_{\underline{k} \in V} \rho(\underline{k}) \otimes \ket{\underline{k}}\bra{\underline{k}} \in L^1(\H), \quad 	\rho \geq 0, \quad \tr(\rho)=1.
\]

\subsection{Quantum trajectories}\label{section:QT} 

The stochastic evolution of the system will depend on the initial state $\rho$ and we shall denote by $\pp_\rho$ the associated probability measure. 
Let us first define the probability space.  We denote by $J=\{1,\dots,v\}$ the set of indices for all possible shifts  in $S=\{s_1,\dots,s_v\} $
and we choose the sample set $\Omega=V \times J^\nn$. 
On $V$ and $J$ we consider the $\sigma$-algebras of the power sets, and on $\Omega$ we then consider the $\sigma$-algebra ${\cal F}$ generated by cylindrical sets. If ${\cal F}_0$ is the power set of $V$ and ${\cal F}_n$ is the $\sigma$-algebra generated by the projections on the first $n$ components of $\Omega$ for $n\ge 1$, notice that $({\cal F}_n)_{n\geq 0}$ is a filtration and ${\cal F}=\sigma\left (\bigcup_{n\geq 0}{\cal F}_n \right)$. 

We define a family of compatible finite dimensional laws which univoquely determines a measure $\pp_\rho$ on $(\Omega, {\cal F})$ by Kolmogorov extension theorem: for all $\underline{k} \in V$, $n \geq 1$, $\underline{j}=(j_1,\dots, j_n) \in J^n$,
\begin{eqnarray*}
\pp_\rho(\{\underline{k}\} \times J^\nn)&=&\tr(\rho(\underline{k})),
\\
\pp_\rho(\{(\underline{k}, \underline{j})\}\times J^\nn)&=&\tr(L_{j_n} \cdots L_{j_1}\rho L_{j_1}^* \cdots L_{j_n}^*).
\end{eqnarray*}

The quantum trajectory is the process $(X_n,\rho_n)_{n\geq 0}$ defined, for $\omega=(\underline{k},j_1,\dots)$, as
\[X_0(\omega)=\underline{k}, \quad \rho_0(\omega)=\frac{\rho(\underline{k})}{\tr(\rho(\underline{k}))},
\]
\[X_{n+1}(\omega)=X_n+s_{j_{n+1}}, \quad \rho_{n+1}(\omega)=\frac{L_{j_{n+1}}\rho_nL_{j_{n+1}}^*}{\tr(L_{j_{n+1}}\rho_nL_{j_{n+1}}^*)}
\qquad \forall n\geq 0.\]
$(X_n,\rho_n)_{n\geq 0}$ is a Markov process on the filtered probability space $(\Omega, {\cal F},{\cal F}_{n\geq 0},\pp_\rho)$, 
with initial law given by 
\[
\pp_\rho \left\{(X_0,\rho_0)=\left (\underline{k},\frac{\rho(\underline{k})}{\tr(\rho(\underline{k}))} \right)\right\}
=\tr(\rho(\underline{k})), 
\quad \underline{k} \in V
\]
and transition probabilities
\[
\pp_\rho \left (X_{n+1}=X_n+s_j,\rho_{n+1}=\frac{ L_j\rho_n L_j^*}{\tr( L_j \rho_n L_j^*)}\bigg\rvert X_n,\rho_n\right )=\tr(L_j\rho_nL_j^*), \quad j \in J, \quad n\geq 1.
\]
Notice that $({\cal F}_n)_n$ coincides with the natural filtration of $(X_n,\rho_n)_{n\geq 0}$. 

In order to fix some ideas about the definition of OQRW and on the behavior of the related position process, we introduce two simple examples, both on the lattice $V=\zz$, for which we provide the simulated trajectories of the rescaled position process in the next figures. In this case ($V=\zz$), the HOQRW has two possible movements at each time step, i.e. v=2, and the walk is completely determined once we fix the two Kraus operators $L_1,L_2$ describing the action on the internal state when moving to the right or to the left. For convenience, we shall call them $R$ and $L$ respectively.

\begin{example}\label{example1}
Let us consider a HOQRW on $V=\zz$ with local space $\hh=\cc^2$ (we denote by $\{e_0,e_1\}$ the canonical basis) and the following local operators:
\[L=\begin{pmatrix} \sqrt{\frac{1}{2}} & 0\\[6pt] 
 -\frac{\sqrt{2}}{3} &\sqrt{\frac{1}{3}}\\[6pt]\end{pmatrix}, \quad
R=\begin{pmatrix} \sqrt{\frac{1}{6}}& 0 \\[6pt]
\frac{1}{3} & \sqrt{\frac{2}{3}} \\[6pt]
\end{pmatrix}
\]
corresponding to going to the left and the right respectively. In this case the local map $\L(\cdot)=L \cdot L^* + R \cdot R^*$ admits a unique invariant state $\tau_0=\ketbra{e_1}{e_1}$. For every initial state $\rho$, simulations show that, for increasing values of $n$, the law of $\frac{X_n-X_0}{\sqrt{n}}$ becomes approximately Gaussian, with fixed variance, and mean growing as $\sqrt n$ (see Figure \ref{fig:ex1}). 
\begin{figure}[h]
     \centering
     \begin{subfigure}[b]{0.8\textwidth}
         \centering
         \includegraphics[width=\textwidth]{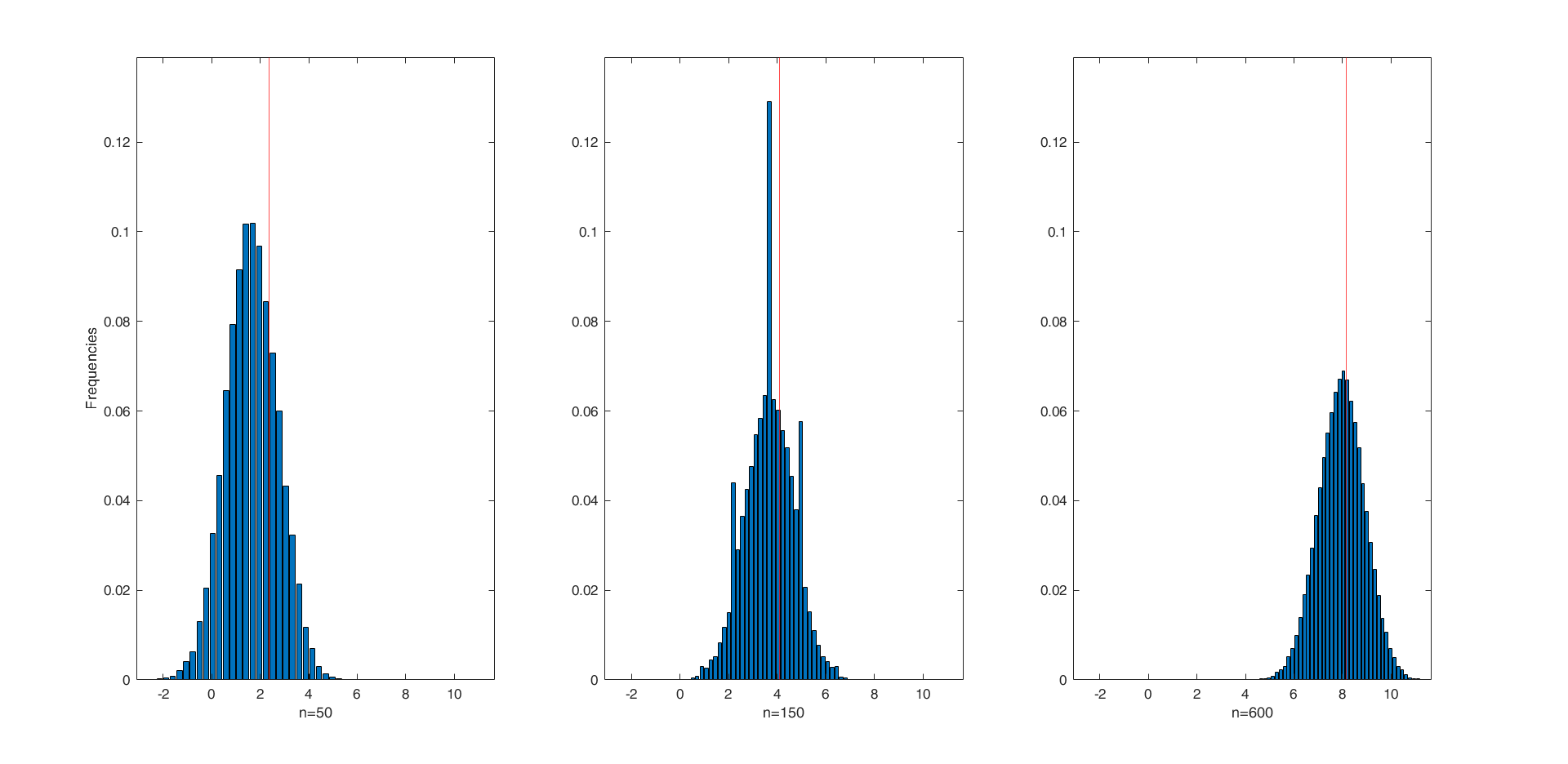}
         \caption{$\rho=\tau_0 \otimes \ketbra{0}{0}$}
         \label{fig:ex1rec}
     \end{subfigure}
     \hfill
     \begin{subfigure}[b]{0.8\textwidth}
         \centering
         \includegraphics[width=\textwidth]{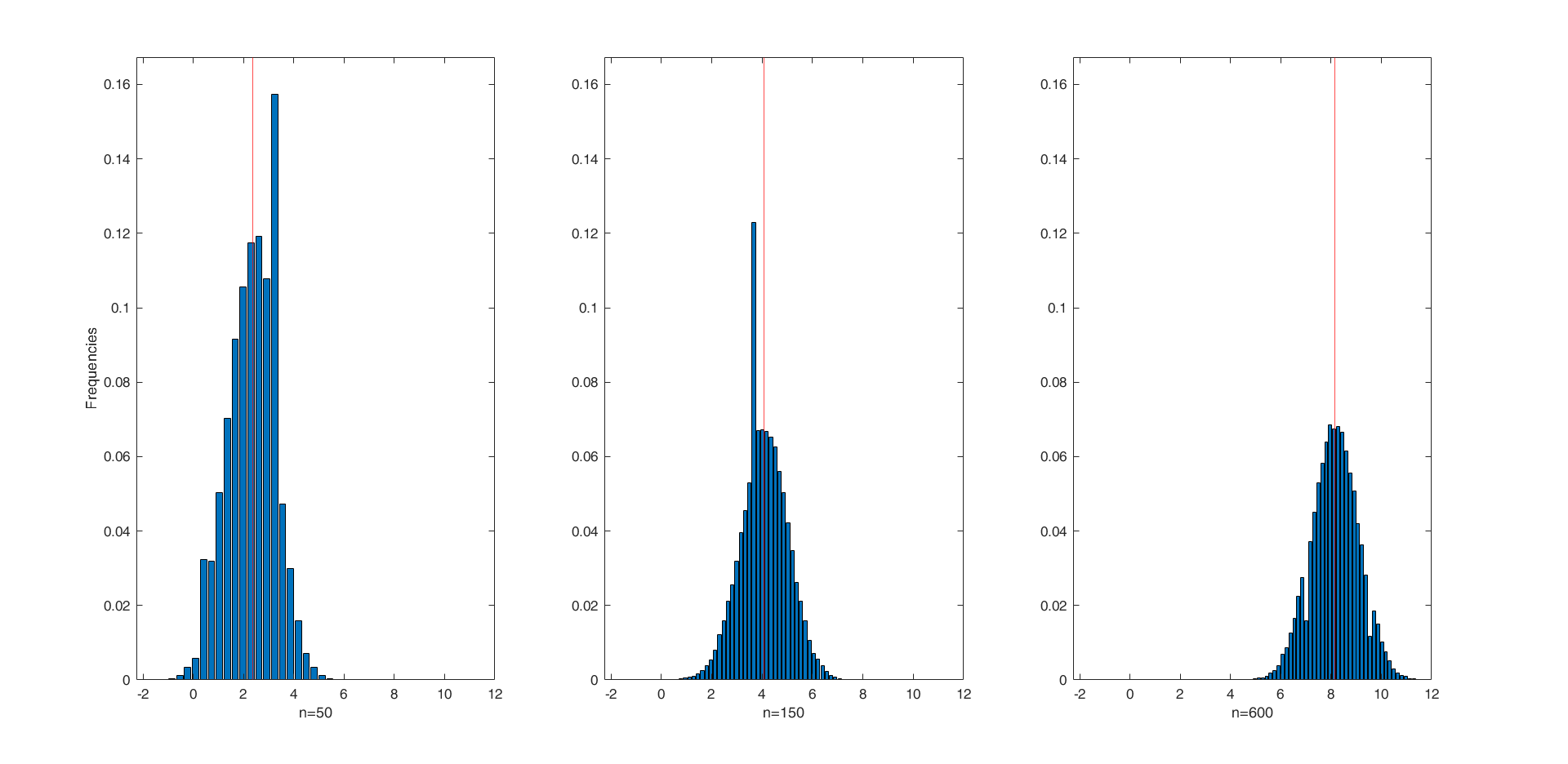}
         \caption{$\rho=\ketbra{e_0}{e_0} \otimes \ketbra{0}{0}$}
          \label{fig:ex1tr}
     \end{subfigure}
      \caption{Panels (a)-(b) show the apperance of the same Gaussian distribution for two different initial states in the model introduced in Example \ref{example1}. We used $N=5 \times 10^{4}$ samples of $\frac{X_{n}}{\sqrt{n}}$ for $n=50,150,600$ in order to draw its profile. The vertical red line corresponds to the mean value of the Gaussian.}\label{1mistura}
\end{figure}  \label{fig:ex1}
\end{example}

\begin{example}\label{example2}
Consider now always $V=\zz$, but local space $\hh=\cc^4$ and local Kraus operators
\[L=\begin{pmatrix}
\frac{1}{2\sqrt{2}}& 0&0&0\\[6pt]
0& \frac{1}{\sqrt{2}}&0&0\\[6pt]
0& 0& \frac{1}{\sqrt{2}}&0\\[6pt]  
-\sqrt{\frac{1}{6}}&0&0&\frac{\sqrt{2}}{3}\\
\end{pmatrix}, \quad
R=\begin{pmatrix}
\sqrt{\frac{3}{8}}& 0&0&0\\[6pt]
0&  \frac{1}{\sqrt{2}}&0&0\\[6pt]
0&0&\frac{1}{\sqrt{2}}&0\\[6pt] 
\frac{1}{\sqrt{3}}& 0&0&\frac{1}{\sqrt{3}}\\[6pt] 
 \end{pmatrix}.
\]
The invariant states of the local map are of the following form: $x\sigma+(1-x)\ketbra{e_3}{e_3}$ where $\sigma$ is any state supported in ${\rm span}\{e_1,e_2\}$ and $x \in [0,1]$. In this case simulations show that, as $n$ increases, the law of $\frac{X_n-X_0}{\sqrt{n}}$ can approach either a Gaussian or the mixture of two Gaussians, whose parameters will be easy to compute using the results of next sections (${\cal N}(0,1)$ and $ {\cal N}(-\sqrt{n}/3,8/9)$). 
Figure \ref{fig:ex2} shows that the  profile of such a mixture and the weights  $a_{\alpha}$ in the previous theorem strongly depend on the initial state.

We shall recover this example in the last section, considering a slightly more general family of Kraus operators.

\begin{figure}[!htbp] 
     \centering
     \begin{subfigure}[b]{0.7\textwidth}
         \centering
         \includegraphics[width=\textwidth]{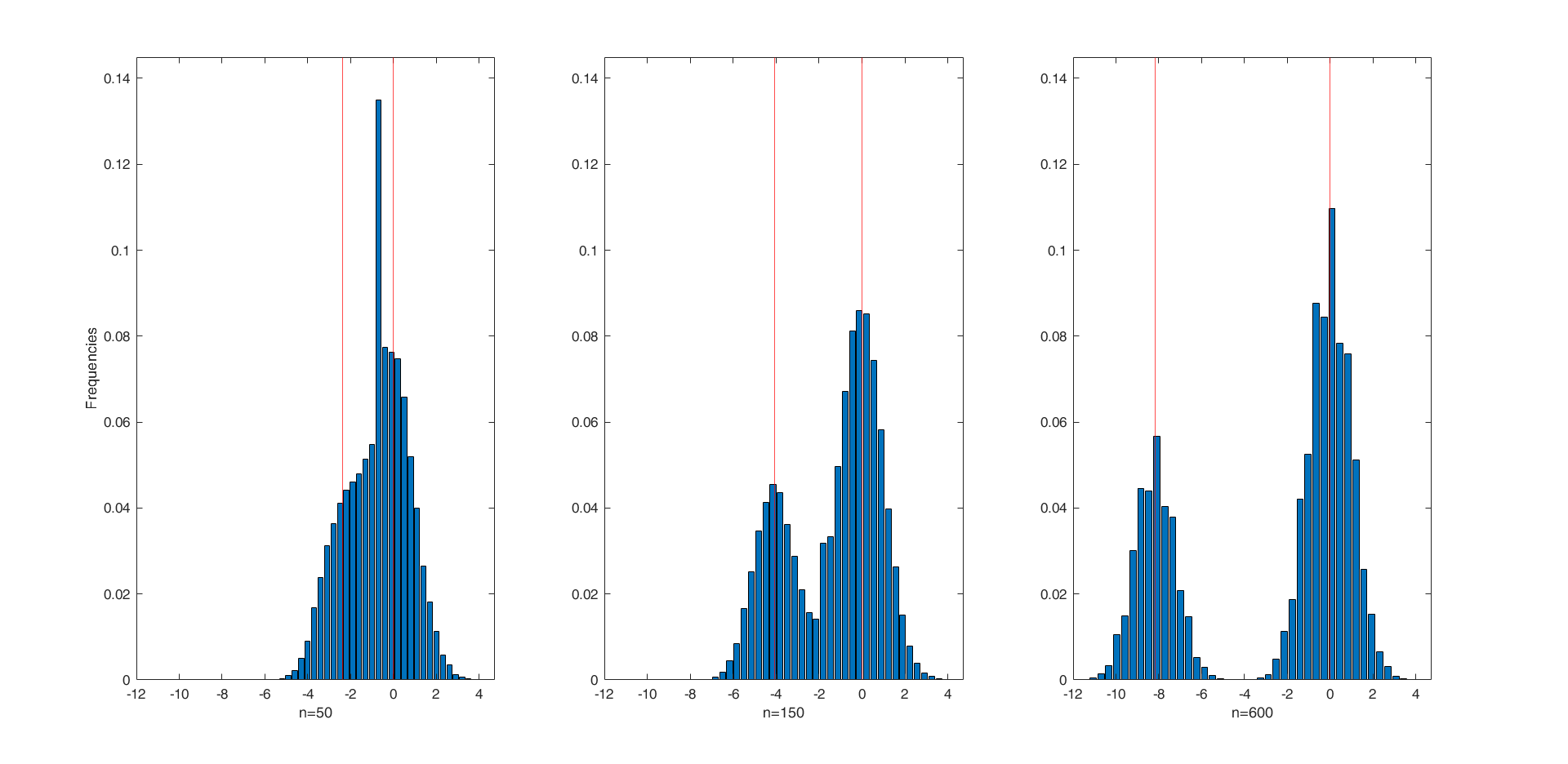}
         \caption{$\rho=\frac{1}{3} (\ketbra{e_1}{e_1}+\ketbra{e_2}{e_2}+\ketbra{e_3}{e_3}) \otimes \ketbra{0}{0}$}
         \label{fig:02}
     \end{subfigure}
     \hfill
     \begin{subfigure}[b]{0.7\textwidth}
         \centering
         \includegraphics[width=\textwidth]{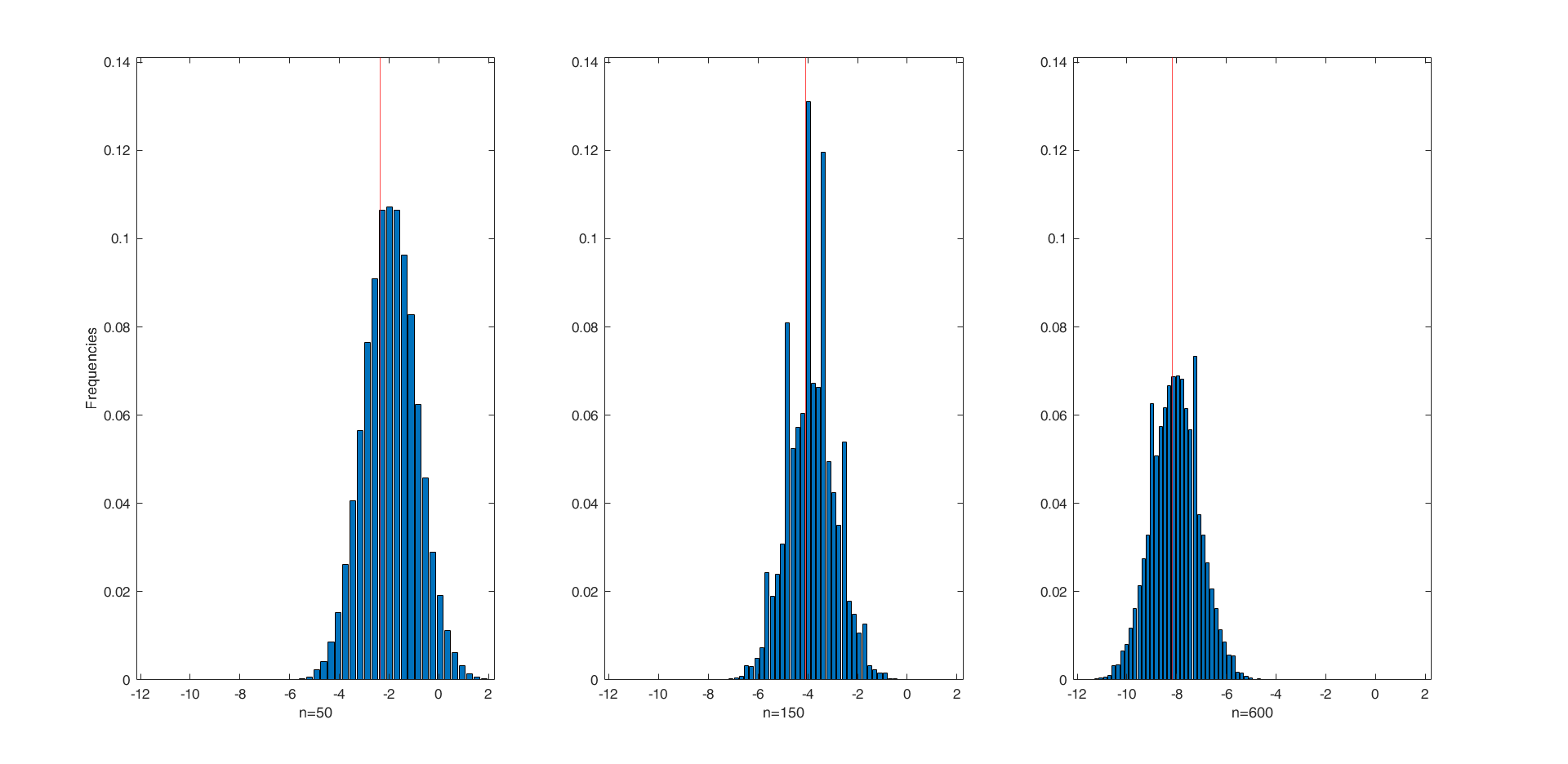}
         \caption{$\rho=\ketbra{e_0}{e_0} \otimes \ketbra{0}{0}$}
          \label{fig:03}
     \end{subfigure}
     \caption{Panel (a) shows a mixture of two Gaussian distributions for a particular choice of the initial state, while panel (b) shows a single Gaussian for another initial state for the model considered in Example \ref{example2}.  We used $N=5\times 10^{4}$ samples of $\frac{X_{n}}{\sqrt{n}}$ for $n=50,150,600$ in order to draw their profile. The vertical red lines correspond to the mean values of the Gaussians.}\label{1mistura}
\end{figure} \label{fig:ex2}
\end{example}

\subsection{Enclosures and absorption}\label{oqw0}
Let $\h$ be a finite dimensional Hilbert space and $\Phi$ be a quantum channel on $L^1(\h)$, the set of trace-class operators on $\h$. Since the topological dual of $L^1(\h)$ is isometrically isomorphic to $B(\h)$, we can define the dual map $\Phi^*:B(\h) \rightarrow B(\h)$ as the operator verifying
\[\tr(\omega\Phi^*(x))=\tr(\Phi(\omega)x) \quad x \in B(\h), \omega \in L^1(\h).
\]
$\Phi^*$ is a completely positive unital (i.e. $\Phi^*(1_\hh)=1_\hh$) bounded operator. 
Given any positive operator $x \in B(\h)$, we define its support projection as the orthogonal projection onto ${\rm supp}(x):=\ker(x)^\perp$.

\begin{definition}
1. A subspace $\V \subset \h$ is said to be an enclosure (or invariant domain) for $\Phi$ if for every positive $\omega \in L^1(\h)$
\[{\rm supp}(\omega) \subset \V \text{ implies } {\rm supp}(\Phi(\omega)) \subset \V.\]
An enclosure will be called minimal if it does not contain other non trivial enclosures. 
\\
2. A quantum channel is called irreducible if the only enclosures are the trivial ones, i.e. $\{0\}$ and $\h$.
\end{definition}
A finite dimensional minimal enclosure is always the support of a unique invariant state of the channel $\Phi$, i.e. a positive trace one operator $\tau$ such that $\Phi(\tau)=\tau$. 

Denoting by $p_\V$ the orthogonal projection onto $\V$, we have the following equivalent characterizations of the notion of enclosure.

\begin{proposition}[see \cite{PC} and references therein]
Let $\Phi$ be a quantum channel acting on $L^1(\h)$. The following are equivalent:
\begin{enumerate}
\item $\V$ is an enclosure;
\item $p_\V$ is a subharmonic projection, i.e. $\Phi^*(p_\V)\geq p_\V$;
\item if $\Phi$ has a representation with Kraus operators $(V_i)_{i\in I}$, 
i.e. $\Phi(\cdot)=\sum_{i\in I} V_i (\cdot) V_i^*$,
then $V_ip_\V=p_\V V_ip_\V$ for every $i \in I$.
\end{enumerate}
\end{proposition}

If we consider the restriction of the quantum channel to $p_\V L^1(\h)p_\V\simeq L^1(\V)$, we obtain again a quantum channel, that we shall denote $\Phi_{|\V}$ with an abuse of notation,
\[ \begin{split}
\Phi_{|\V}:  p_\V L^1(\h)p_\V &\rightarrow p_\V L^1(\h)p_\V \\
  p_\V \sigma p_\V & \mapsto p_\V\Phi( p_\V \sigma p_\V)p_\V=\Phi( p_\V \sigma p_\V); \\
\end{split}
\]
its dual map $\Phi^*_{|\V}$ acts on $p_\V B(\h)p_\V \simeq B(\V)$ and for every $x \in B(\V)$
\[\begin{split}
 \Phi^*_{|\V} ( p_\V x p_\V) = p_\V \Phi^*(p_\V x p_\V)p_\V=p_\V \Phi^*( x )p_\V.
\end{split}
\]

The channel $\Phi_{|\V}$ restricted to a minimal enclosure is trivially irreducible by construction. 

Given an enclosure $\V$, we can define the associated absorption operator (see \cite{CG}) as
\begin{equation} \label{eq:absdef}
\AV:=\lim_{n \rightarrow +\infty} \Phi^{*n}(p_\V).
\end{equation}
Absorption operators enjoy remarkable properties that we recall below.
\begin{proposition}[{\cite[Proposition 4]{CG}}] \label{prop:abs}
The following statements hold true:
\begin{enumerate}
\item $0 \leq \AV \leq 1$;
\item $\AV$ is harmonic, that is $\Phi^*(\AV)=\AV$;
\item $\ker(\AV)$ is an enclosure;
\item $A(\V)=p_\V + p_{\V^\perp} A(\V) p_{\V^\perp}$.
\end{enumerate}
\end{proposition}

Some additional discussion about the general structure of enclosures for quantum channels and possible decompositions of $\h$ in orthogonal enclosures will be developed in Section \ref{sec:mixgauss}.

In the same way as for positive matrices, there exists a Perron-Frobenius theorem for quantum channels. We denote the spectral radius of a map $\Phi$ as
$r(\Phi):=\sup\{\vert\lambda \vert:\lambda\in \rm{Sp}( \Phi)\}$,
where $\rm{Sp}(\Phi)$ is the spectrum of $\Phi$.

\begin{theorem}[{\cite[Theorems 6.4 and 6.5]{Wo}}]\label{pf0}
Let $\Phi$ be a positive map acting on $L^1(\h)$; then $r(\Phi)$ is a eigenvalue and the corresponding eigenvector is positive. If in addition $\Phi$ is irreducible, then $r(\Phi)$ is geometrically simple and the corresponding eigenvector is strictly positive.
\end{theorem}

In case $\Phi$ is quantum channel $1=r(\Phi)=\lVert \Phi \rVert_\infty$.

\section{Selecting a single Gaussian} \label{sec:singlegauss}

In this section, we shall concentrate on a single minimal enclosure $\V$ of the local channel $\L$ and we shall introduce a proper associated probability measure $\pp'_\rho$, which is absolutely continuous with respect to $\pp_\rho$, with a relative density which assigns weights according to the absorption in $\V$. We shall prove that the position process $(X_n)_n$ always satisfies a central limit theorem under this measure (Theorem \ref{th:singleCTL}). Previous CLT results can be seen as a consequence of the case of a single enclosure (see Remark \ref{rmk:previous} below). 

According to the notations introduced in the previous section, we shall call $\AV=\lim_{n \rightarrow +\infty} \L^{*n}(p_\V)$ the absorption operator of the enclosure $\V$ for $\L $ and we denote by ${\tilde p}_\V$ the support projection  of $\AV$.

\begin{lemma} \label{lem:supermartingale}
Let $(Y_n)_{n\ge 0}$ be the process defined by $Y_n=\tr((\AV\rho_n))$ for any $n\ge 0$.
\begin{enumerate}
\item Then $(Y_n)_{n\ge 0}$ is a positive and bounded $\pp_\rho$-martingale, converging (almost surely and $L^1$) to a random variable $Y_\infty$ valued in $[0,1]$. 
\item
If $\mathbb{E}_{\rho}[Y_0]=\mathbb{E}_\rho[\tr(\AV\rho_0)]>0$, 
we can define a new probability measure $\pp'_\rho$ such that
\[
\frac{d\pp'_\rho}{d\pp_\rho}=\frac{Y_\infty }{\mathbb{E}_\rho[Y_0]},
\qquad
\frac{d\pp'_\rho}{d\pp_\rho} \bigg\rvert_{{\cal F}_n}=\frac{Y_n }{\mathbb{E}_\rho[Y_0]}.
\]
Moreover the density $\frac{d\pp'_\rho}{d\pp_\rho}$ is valued in $[0,{\mathbb{E}_\rho[Y_0]}^{-1}]$ and
\begin{equation}
 \begin{split}
 \label{eq:absdensity}
&\left \{\frac{d\pp'_\rho}{d\pp_\rho}=\frac{1}{\mathbb{E}_\rho[Y_0]} \right \}=\left \{\lim_{n\rightarrow + \infty} \lVert p_\V \rho_n p_\V-\rho_n \rVert=0 \right\}, \\
& \left \{\frac{d\pp'_\rho}{d\pp_\rho}=0 \right \}=\left \{\lim_{n\rightarrow + \infty} \lVert \tilde{p}^\perp_\V \rho_n \tilde{p}^\perp_\V-\rho_n \rVert=0 \right\}.\\
\end{split}
\end{equation}
\end{enumerate}
\end{lemma}

We remark that, for this lemma, it is not necessary for $\V$ to be minimal.
The last sentence of the statement gives a mathematical meaning to the intuition that, given any enclosure $\V$, the corresponding $\pp^\prime_\rho$ encodes the notion of conditioning to the ``absorption in $\V$''. Nevertheless,  $Y_\infty$ is not a Bernoulli random variable in general, hence there does not need to exist any measurable set $B \in {\cal F}$ such that $\pp^\prime_\rho(\cdot)=\pp_\rho(\cdot|B)$, even if it can happen in some cases (see Example \ref{Example2} and in particular the simulations in Figure \ref{fig:traj}).

\begin{proof}
$Y_n$ is trivially positive and bounded and
\[\mathbb{E}_{\rho}[Y_{n+1}|{\cal F}_n]
=\sum_{i=1}^{v} \tr(L_i \rho_n L_i^*) \frac{ \tr(\AV L_i \rho_n L_i^*)}{ \tr(L_i \rho_n L_i^*)}
=\tr(\L^*(\AV)\rho_n )= \tr(\AV\rho_n)=Y_n.
\]

Since $(Y_n)$ is a positive and bounded martingale, it converges almost surely and in $L^{1}:=L^1(\Omega,\pp_\rho)$ to a positive random variable $Y_\infty$. When $\V$ and $\rho$ are such that 
$\mathbb{E}_\rho[\tr(\AV\rho_0)]=\sum_{\underline{k} \in V} \tr(\AV \rho(\underline{k}))>0$, we can introduce the random variables
\[0\leq Z_n:=\frac{Y_n }{\mathbb{E}_\rho[\tr(\AV\rho_0)]},
\]
and the sequence $(Z_n)$ is a $\pp_\rho$-martingale with expected value equal to $1$ and converges almost surely to
\begin{equation} \label{eq:density}
0 \leq Z_\infty:=\frac{Y_\infty }{\mathbb{E}_\rho[\tr(\AV\rho_0)]}.
\end{equation}
 Note that $Z_{\infty}\in L^1$. Therefore we can consider the new measure $\pp'_\rho$ which has density $Z_\infty$ with respect to $\pp_\rho$, so that
\[\frac{d\pp^\prime_\rho}{d\pp_\rho}=Z_\infty, \qquad 
\frac{d\pp_\rho^\prime}{d\pp_\rho} \bigg\rvert_{{\cal F}_n}=Z_n.
\]
The range of $\frac{d\pp'_\rho}{d\pp_\rho}$ trivially follows from the fact that $0 \le Y_\infty \le 1$. We postpone the proof of relations \eqref{eq:absdensity} to Section \ref{sec:mixgauss}, since we need some notions that we will introduce later on.
\end{proof}


We shall use spectral analysis and deformation techniques in order to prove the central limit theorem for the position process $(X_n)_{n \ge 0}$ under the measure $\pp^\prime_\rho$.
For all $u \in \rr^d$, let us define the following operators:
\[{L}^{(u)}_i=e^{\frac{u \cdot s_i}{2}}{L}_i,\qquad
\tilde{L}^{(u)}_i=e^{\frac{u \cdot s_i}{2}}{\tilde p}_\V {L}_i {\tilde p}_\V,  \quad i=1,\dots,v
\]
and we call $\L_u$ and $\tilde{\L}_u$ the analytic perturbations of $\L$ and $\tilde{\L}=\tilde{\L}_0$ respectively, defined as the completely positive operators
\begin{align*}
\L_u(\sigma) & =\sum_{i=1}^v L^{(u)*}_i \sigma L^{(u)}_i , \qquad 
\tilde{\L}_u (\sigma)  = \sum_{i=1}^v \tilde{L}^{(u)*}_i \sigma \tilde{L}^{(u)}_i.
\end{align*}
We denote by $\lambda_u$ the spectral radius of $\tilde{\L}_{u}$, that is $\lambda_u=r(\tilde{\L}_{u})$.
Theorem \ref{pf0} ensures that $\lambda_u \in \rm{Sp}(\tilde{\L}_{u})$ with corresponding positive eigenvector $\tau_u$. Notice that $\lambda_0=1$ and $\tau_0$ is the unique minimal invariant state supported on $\V$. Moreover, $\L_u$ and $\tilde{\L}_u$ can be extended for complex values of $u$ and form two analytic families of matrices: $\L_u(\sigma) =\sum_{i=1}^v {\rm e}^{u \cdot s_i} L^*_i\sigma L_i$ and $\tilde{\L}_u(\sigma) =\sum_{i=1}^v {\rm e}^{u \cdot s_i} {\tilde p}_\V L^*_i {\tilde p}_\V \sigma  {\tilde p}_\V {L}_i {\tilde p}_\V$. In Lemma \ref{lem:analytic} we shall prove that also the perturbed eigenvalue $\lambda_u$ and eigenvector $\tau_u$ are analytic at least in a neighborhood of the origin.

Notice that all previous mathematical objects depend on the enclosure $\V$, so it would be more precise to highlight this and denote them $\tilde{L}^{(u,\V)}_i, \tilde{\L}_u^\V ,...,  \pp^{\prime(\V)}_\rho$. Since the notations are already quite heavy, we drop the dependence on $\V$ in this section, since we shall use only one enclosure and we shall recover it when necessary, treating the general case.

\begin{lemma} \label{lem:analytic}
Let $\V$ be a minimal enclosure. The operators
$\tilde{\L}$ and $\tilde{\L}_{|\V}=\L_{|\V}$ have the same peripheral eigenvalues and eigenvectors with the same multiplicities.

Moreover in a complex neighborhood of the origin the following hold true:
\begin{enumerate}\item $u \mapsto \lambda_u$ and $u \mapsto \tau_u$ are analytic;
\item ${\rm supp}(\tau_u) \subset \V$.
\end{enumerate}
Hence $\lambda_u$ and $\tau_u$ coincide with the analogous quantities for the restricted deformation
$\tilde{\L}_{u|\V}={\L}_{u|\V}$ 
(i.e. $\lambda_u:=r({\L}_{u|\V}), {\L}_{u|\V}(\tau_u)=\lambda_u \tau_u$).
\end{lemma}

\begin{proof}
1. Let $\theta \in [0,2\pi)$ and $\sigma \in L^1(\h)$ such that
\begin{equation} \label{eq:eigen}
\tilde{\L}(\sigma)={\rm e}^{i\theta}\sigma.
\end{equation}
In order to prove that the peripheral eigenvectors and eigenvalues of $\tilde{\L}$ are the same as those of $\tilde{\L}_{|\V}$ we need to prove that $\sigma=p_\V \sigma p_\V$. Let us consider the orthogonal decomposition ${\rm supp}(\AV)=\V \oplus {\cal W}$; by definition ${\cal W}={\rm supp}(\AV- p_\V)$ and, since ${\rm dim}(\h)<+\infty$, we know that there exists a constant $\gamma>0$ such that $p_{\cal W} \leq \gamma (\AV-p_\V)$, hence by \cite[Theorem 14]{CG} we have that
\[\tilde{\L}^{*n}(p_{\cal W})=\tilde{p}_\V\L^{*n}(p_{\cal W})\tilde{p}_\V \leq \gamma \tilde{p}_\V\L^{*n}(\AV-p_\V)\tilde{p}_\V \rightarrow 0.
\]
This implies that $\lim_{n \rightarrow +\infty}\lVert \tilde{\L}^n(\sigma)-p_\V \tilde{\L}^n(\sigma) p_\V \rVert = 0$, which, together with equation \ref{eq:eigen}, implies that $\sigma=p_\V\sigma p_\V$. If we consider $\sigma$ as above and $\xi$ is such that $\tilde{\L}(\xi)=e^{i\theta}\xi + \sigma$, with the same reasoning as before we can deduce that also $\xi=p_\V \xi p_\V$ and hence the algebraic multiplicity of ${\rm e}^{i \theta}$ is the same for $\tilde{\L}$ and $\tilde{\L}_{|\V}$. \\
2. By perturbation theory of linear matrices (see \cite{Ka}), we only need to show that $\lambda_0=1$ is an algebraically simple eigenvalue of $\tilde{\L}$, which, by virtue of point 1, is equivalent to prove it for $\tilde{\L}_{|\V}=\L_{|\V}$ and this follows for instance from \cite[Proposition 6.2]{Wo}.\\
3. Notice that by definition $\tilde{\L}_{u}$ preserves the set $p_\V L^1(\tilde{\hh})p_\V$ and eigenvalues and eigenvectors of $\tilde{\L}_{u|\V}$ are also eigenvalues and eigenvectors of $\tilde{\L}_{u}$. Let $\lambda_u^\V$ be the perturbation of $1$ for $\tilde{\L}_{|\V}$; by \cite[Theorem VII.1.7]{Ka} and the proof of point 2 in the present Lemma, for small values of $u$, $\lambda_u$ is the unique eigenvalue of $\tilde{\L}_{u}$ in a neighborhood of $1$ and it is algebraically simple, however $\lambda_u^\V$ is another eigenvalue of $\tilde{\L}_{u}$ and $\lambda_0^\V=1$ too, hence they must coincide in a neighborhood of the origin (remember that $u \mapsto \lambda_u^\V$ is continuous, see \cite[Theorem 5.1]{Ka}). Therefore we have that ${\rm supp}(\tau_u) \subset \V$.
\end{proof}

In Theorem \ref{th:singleCTL} below, we shall apply Bryc's theorem to prove the central limit theorem for the position process. We quote it for the reader's convenience:

\begin{theorem}[Bryc, {\cite[Proposition 1]{Br}}]\label{theo:Bryc}
Let $(T_n)_{n\ge 0}$ be a sequence of random variables defined on the probability spaces $(\Omega_n, {\cal F}_n,\pp_n)$, $T_n: \Omega_n \rightarrow \rr^d$ and suppose there exists $\epsilon>0$ such that
\[h(u)=\lim_{n \rightarrow +\infty} \frac{1}{n} \log(\mathbb{E}_n[{\rm e}^{u\cdot T_n}])
\]
exists for every complex $u$ with $|u|<\epsilon$. Then
\[\frac{(T_n-\mathbb{E}_n[T_n])}{\sqrt{n}}{\, \longrightarrow\, } 
{\cal N}(0,D)
\qquad{\text{(in law)}},
\]
 where  ${\cal N}(0,D)$ denotes a centered Gaussian measure with covariance $D=H(h)(0) \geq 0$ ( H is the hessian of $h$ at $u=0$), and 
 
 \[ \lim_{n \rightarrow +\infty}\frac{\mathbb{E}_n[T_n]}{n}=\nabla h(0).
 \]
\end{theorem}
The gradient and the hessian of the limit function $h$ will then describe the asymptotic mean and covariance matrix and the following lemma proves that they are related to the spectral radius of the perturbed operator restricted to the minimal enclosure $\V$.

\begin{lemma}\label{lambda} 
The function $c: \rr^d \ni u \mapsto \log(\lambda_u)$ is infinitely differentiable in $0$. For every $u \in \rr^d$, we introduce the operators $\L^\prime_{|\V,u}$ and $\L''_{|\V,u}$ by
\[
{\L'_{|\V,u}}, {\L''_{|\V,u}}:  L^1(\V) \longrightarrow L^1(\V)
\]
\[{\L^\prime_{|\V,u}}(\sigma) = \sum_{i=1}^v u \cdot s_i L_i \sigma  L_i^*, 
\qquad {\L''_{|\V,u}}(\sigma)= \sum_{i=1}^v (u \cdot s_i)^2L_i \sigma  L_i^*.
\]
Denoting $\lambda^\prime_u=\frac{d\lambda_{tu}}{dt}\big\rvert_{t=0}$, $\lambda^{\prime \prime}_u=\frac{d^2\lambda_{tu}}{dt^2}\big\rvert_{t=0}$, we have
\begin{align*}
\lambda^\prime_u=\tr({\L'_{|\V,u}}(\tau_0)) , \qquad
\lambda''_u=\tr({\L''_{|\V,u}}(\tau_0))+2\tr({\L'_{|\V,u}}(\eta_u)) 
\end{align*}
where $\eta_u \in L^1(\V)$ is the unique solution with zero trace of the equation
\begin{equation*} 
(\id-\L_{|\V})(\eta_u)={\L^\prime_{|\V,u}}(\tau_0)-\tr({\L^\prime_{|\V,u}}(\tau_0))\tau_0.
\end{equation*}
This implies immediately that

\begin{equation*} 
dc(0)(u)=\lambda^\prime_u,\quad d^2c(0)(u)=\lambda^{\prime\prime}_u-{\lambda^\prime_u}^2.
\end{equation*}

\end{lemma}

\begin{proof}
Notice that
\[{\L^\prime_{|\V,u}}(\sigma) = \sum_{i=1}^v u \cdot s_i \{p_\V L_i p_\V \}\sigma \{p_\V L_i^* p_\V\},
\]
due to the fact that $\V$ is an enclosure 
(and similarly for ${\L''_{|\V,u}}(\sigma)$). This fact, together with Lemma \ref{lem:analytic}, allows us to reduce the analysis to the irreducible channel $\L_{|\V}$ and the proof is the same as in \cite[Corollary 5.9]{CP}.
\end{proof}

\begin{theorem} \label{th:singleCTL}
Consider a minimal enclosure $\V$, and $\tau_0$ and $\lambda_u$ defined as before.
We introduce the vector
\[
m = \sum_{i=1}^v \tr(L_i \tau_0 L_i^*)s_i
\]
and the matrix $D$ which is the unique matrix satisfying the following formula for every $u \in \rr^d$:
\begin{equation*}
\langle u,Du \rangle=\lambda^{\prime\prime}_u-{\lambda^\prime_u}^2.
\end{equation*}
Then, under $\pp^\prime_\rho$,
\begin{equation} \label{CLT}
\frac{(X_n-X_0)-nm}{\sqrt{n}} \rightarrow {\cal N}(0,D)
\end{equation}
 where the convergence is in law. Moreover
\begin{equation*}
\left\vert \frac{\mathbb{E}^\prime_\rho[X_n-X_0]}{n} -m \right\vert =O\left (\frac{1}{n}\right).
\end{equation*}
\end{theorem}
\begin{remark} \label{rmk:previous}
We point out that, when there is a unique minimal enclosure $\V$, then $A(\V)=1_\hh$, $\pp^\prime_\rho=\pp_\rho$, and Theorem \ref{th:singleCTL} is the central limit theorem for the position process. 
\end{remark}
\begin{proof}
With some calculation we get the expression for the moment generating function of the process $(X_n-X_0)$ for every $u \in \cc^d$:
\[\begin{split}
\mathbb{E}^\prime_\rho[e^{u\cdot ( X_n-X_0)}]&=\frac{1}{\mathbb{E}_\rho[\tr(\AV\rho_0)]}\sum_{\underline{k} \in V} \sum_{s_{j_1}, \dots, s_{j_n}} e^{u \cdot \sum_{k=1}^n s_{j_k}} \tr(A(\V) \tilde{L}_{j_n} \cdots  \tilde{L}_{j_1} \rho(\underline{k})  \tilde{L}_{j_1}^* \cdots  \tilde{L}_{j_n}^*)=\\
&=\sum_{\underline{k}\in V}\frac{\tr(\AV \tilde{\L}_u^n(\rho(\underline{k})))}{\mathbb{E}_\rho[\tr(\AV\rho_0)]}.
\end{split}
\]
We are interested in the functions of the form
\[h_n(u)=\frac{1}{n} \log (\mathbb{E}_\rho^\prime[e^{\langle u, X_n-X_0 \rangle}]).
\]
In order to apply Bryc's theorem (Theorem \ref{theo:Bryc}), we need to show the existence of $\lim_{n \rightarrow +\infty} h_n(u)$ for $u$ in a complex neighborhood of $0$. Let us first consider the case where $\tilde{\L}_{|\V}$ is aperiodic (since we mimic the proof of \cite[Theorem 5.12]{CP}, we refer to \cite{CP} for more information about the notion of period for quantum channels). In this case we have
\[\delta=\sup\{|\lambda| : \lambda \in {\rm Sp}(\tilde{\L})\setminus \{1\}\}<1
\]
and so, considering the Jordan form of $\tilde{\L}$, there exists $\epsilon>0$ such that $\delta+\epsilon<1$ and for $u$ in a neighbourhood of $0$, for $n \in \nn$ we have
\[\tilde{\L}^n_u(\cdot)=\lambda_u^n(\varphi_u(\cdot)\tau_u + O((\delta+\epsilon)^n))
\]
where $\varphi_u$ is a linear form on $L^1(\hh)$, analytic in $u$ in the considered neighbourhood of the origin and $O$ is with respect to any norm (remember that in finite dimension all the operator norms are equivalent). Therefore
\[\begin{split}
h_n(u)&=\log(\lambda_u)\\
&\quad +\frac{1}{n}\left [-\log(\mathbb{E}^\prime_\rho[\tr(\AV\rho_0)])+\log \left (\sum_{\underline{k} \in V} \varphi_u(\rho(\underline{k}))\tr(A(\V)\tau_u) + O((\delta+\epsilon)^n\right )\right ]\\
&\rightarrow \log(\lambda_u).\\
\end{split}
\]
From the proof of Theorem \ref{theo:Bryc} we know that all $h_n$ are analytic in a neighborhood of the origin. Further, these functions converge uniformly on compact sets to $h$ and $\sup_{u \in K}|h_n(u)-\log(\lambda_u)|=O(1/n)$ where $K$ is a compact set in the considered neighborhood of the origin. Hence, by Cauchy integral formula we can deduce, since
\[\frac{\mathbb{E}^\prime_\rho [X_n-X_0]}{n}=\nabla h_n(0) \text{ and } m=\nabla h(0)
\]
that
\[\left |\frac{\mathbb{E}^\prime_\rho[X_n-X_0]}{n}-m \right |=O \left (\frac{1}{n} \right )
\]
and this allows us to put $nm$ instead of ${\mathbb{E}^\prime_\rho[X_n-X_0]}$ in equation \eqref{CLT}.

On the other hand, if $\L_{|\V}$ has period $l>1$ with cyclic resolution $p_0,\dots,p_{l-1}$, we can write for $n=ql+r$ and $0\leq r <l$
\[\mathbb{E}^\prime_\rho[e^{u \cdot ( X_n-X_0 )}]=\sum_{j=0}^{l-1}\underbrace{\sum_{\underline{z}\in V}\frac{\tr(A(p_j)\rho(\underline{z}))}{\mathbb{E}_\rho[\tr(\AV\rho_0)]}}_{w_j}
\underbrace{\sum_{\underline{k} \in V}\frac{\tr(A(p_j)\tilde{\L}_u^n(\rho(\underline{k})))}{\sum_{\underline{z}\in V}\tr(A(p_j)\rho(\underline{z}))}}_{II}.
\]
We can safely define $A(p_j)$ using $\tilde{\L}^l$, for which $p_0,\dots,p_{l-1}$ are minimal enclosures. Furthermore we can express $II$ as
\[II=\sum_{\underline{k} \in V}\frac{\tr(A(p_j)\tilde{\L}_u^{lq}(\tilde{\L}_u^r(\rho(\underline{k}))))}{\sum_{\underline{z}\in V}\tr(A(p_j)\rho(\underline{z}))}.
\]
The support projection of $A(p_j)$, which we call $P_j$, is superharmonic for $\tilde{\L}^l$, hence, if we consider $\tilde{\L}^l_{j,u}:=P_j\tilde{\L}^l(P_j \cdot P_j)P_j$, we can write 
$$
\tr(A(p_j)\tilde{\L}_u^{lq}(\tilde{\L}_u^r(\rho(\underline{k}))))=\tr(A(p_j)\tilde{\L}^{lq}_{j,u}(P_j\tilde{\L}_u^r(\rho(\underline{k}))P_j))
$$ 
and we are back to the aperiodic case. Furthermore the perturbation of $1$ for every reduction $\tilde{\L}^l_{j,u}$ is the same as the one of $\tilde{\L}_u^l$ since $P_j\tau_uP_j$ is an eigenvector of $\tilde{\L}^l_{j,u}$ for the eigenvalue $\lambda^l_u$:
\[\tilde{\L}^l_{j,u}(P_j\tau_uP_j)=P_j\tilde{\L}^l_u(\tau_u)P_j=\lambda^l_u P_j\tau_u P_j.
\]
Therefore we can again prove the statement.
\end{proof}

\section{General case: mixture of Gaussians} \label{sec:mixgauss}

In order to tackle the general case, we now need to consider different enclosures and to handle the simultaneous appearance  of different Gaussians. The description of the general context requests the introduction of some additional notions in order to describe an appropriate decomposition of the local Hilbert space $\h$. This will induce a decomposition of the measure ${\mathbb P}_\rho$ in terms of  measures of the form ${\mathbb P}^\prime_\rho$ as defined  in Lemma \ref{lem:supermartingale}.

{\bf Decomposition of the local Hilbert space and of the recurrent subspace.}

We introduce the fast recurrent and the transient space for the local map following \cite{BN,Uma}; for other notions of recurrence for OQRWs we refer to \cite{BBP,DhMu,JL} and references therein. We denote by $\R$ the fast recurrent space for the channel $\L$
\begin{align}\label{rec}
\R=\sup \{\text{$\textrm{supp}(\omega) \text{ }\vert$ $\omega$ is an invariant state for $\L$} \}.
\end{align}
$\R$ is an enclosure and, when \underline {the space $\h$ is finite dimensional}, any minimal enclosure is included in $\R$ and is the support of a unique extremal invariant state; moreover we have trivial slow recurrent subspace, while $\R$ is always non trivial and ``absorbing''. Further, the orthogonal complement of $\R$ is the transient space, usually denoted by $\TT$ and the absorption in $\R$ is the identity operator
(see \cite{BN,CP,Uma})
\[
\hh=\R\oplus\TT,\qquad
A(\R)=1_\hh-\lim_{n \rightarrow +\infty} \L^{*n}(p_\TT)= 1_\hh.
\]

The structure of quantum channels induces a decomposition of the fast recurrent space, also naturally related to the invariant states (see \cite{BN} for the finite dimensional case and \cite{PC,JP,Uma} for infinite dimensional state spaces).
This decomposition is the noncommutative counterpart of the decomposition in communication classes for classical Markov chains and plays a fundamental role in different contexts. Here we shall briefly recall the decomposition and the main properties we need. 

For a quantum channel acting on $L^1(\h)$, there exists a unique decomposition of $\R$ of the form 
\begin{align*}
\R=\underset{\alpha\in A} {\oplus}\, \chi_\alpha,
\end{align*}
where $(\chi_\alpha)_{\alpha\in A}$ is a finite set of mutually orthogonal enclosures and 
every $\chi_\alpha$ is minimal in the set of enclosures verifying the property:

\centerline{for any minimal enclosure ${\cal W}$ either ${\cal W}\perp \chi_\alpha$ or ${\cal W} \subset \chi_\alpha$.}

Every $\chi_\alpha$ either is a minimal enclosure or can be further decomposed (but not in a unique way!) as the sum of mutually orthogonal isomorphic minimal enclosures, i.e. 
\begin{equation}\label{R-decomposition}
\chi_\alpha= \underset{\beta \in I_\alpha}{\oplus} \V_{\alpha,\beta}, \qquad
\R=\underset{\alpha\in A} {\oplus} \, \chi_\alpha = \underset{\alpha\in A} {\oplus} \,\underset{\beta \in I_\alpha}{\oplus}\V_{\alpha,\beta},
\end{equation} 
for some finite set $\V_{\alpha,\beta}, \beta\in I_\alpha$ of minimal enclosures and, if we fix a particular $\bar\beta\in I_\alpha$, there exists a unitary transformation $U_\alpha$ such that
\begin{equation}\label{unitary}
 U_\alpha :{\mathbb C}^{|I_\alpha|} \otimes  \V_{\alpha,\bar \beta}
\rightarrow \chi_\alpha.
\end{equation}
Moreover one can define an irreducible quantum channel $\psi$ on $B(\V_{\alpha,\bar \beta})$ which completely describes the restriction of the channel to $\chi_\alpha$
\begin{equation} \label{psi}
\L_{|\R}^*(U_\alpha (a\otimes b)U^*_\alpha)
= U_\alpha (a\otimes \psi(b))U^*_\alpha
\qquad a\in B({\mathbb C}^{|I_\alpha|}), \,  b\in B(\V_{\alpha,\bar \beta}).
\end{equation}

\begin{remark}
$\chi_\alpha$ is a minimal enclosure if and only if $|I_\alpha|=1$. Otherwise, it is not minimal and it admits infinite possible decompositions in orthogonal minimal enclosures of the form $U_\alpha (\cc v \otimes \V_{\alpha,\overline{\beta}})$ for $v \in \cc^{|I_\alpha|}$. In this case, however, a rigid structure of the channel essentially reduces the action on any minimal enclosure inside $\chi_\alpha$ to be the same up to a unitary transform.
\end{remark}

\begin{lemma} \label{lemma-m-D}
The parameters $m=m(\V)$ and $D=D(\V)$ introduced in Theorem \ref{th:singleCTL} are independent of the particular minimal enclosure $\V$ in $\chi_\alpha$.  Then we define
\[m_\alpha:=\sum_{i=1}^v \tr(L_i \tau_0^\V L_i^*)s_i,
\qquad \langle u, D_\alpha u\rangle = \lambda''_u - {\lambda'_u}^2,
\]
where $\lambda'_u,\lambda''_u$ are defined as in Lemma \ref{lambda} for $\L_{|\V}$.
\end{lemma}

\begin{proof}
Let us consider two minimal enclosures 
$\V$ and ${\mathcal W}$ contained in a same $\chi_\alpha$. We just have to prove that the parameters $m$ and $D$ are equal for the two enclosures.

Relations \eqref{unitary} and \eqref{psi} imply that there exist two vectors $v,w$ in ${\mathbb C}^{|I_\alpha|}$ such that 
$$
\V = U_\alpha (({\mathbb C}v)\otimes \V_{\alpha,\bar \beta})U^*_\alpha, \qquad
{\mathcal W} = U_\alpha (({\mathbb C}w)\otimes \V_{\alpha,\bar \beta})U^*_\alpha, \qquad
$$
and we can define a partial isometry $Q  = U_\alpha ((|w\rangle \langle v |)\otimes 1_{\V_{\alpha,\bar \beta}})U^*_\alpha$, from $\V$ to ${\mathcal W}$, such that
\begin{align}\label{op Q}
Q^{\ast}Q=p_{\V}, \quad QQ^{\ast}=p_{\mathcal W}
\quad\text{ and } \quad
\L_{|\V}^*(x)
= Q^*\L_{|{\mathcal W}}^*(Q x Q^*) Q \qquad \forall x \in B(\V),
\end{align}
where $\L_{|\V}$ and $\L_{|{\mathcal W}}$ are the restrictions of $\L$ to $\V$ and ${\mathcal W}$ respectively, following the notations introduced before. Due to relation \eqref{psi}, $Q$ (and $Q^*$) is also a fixed point for the dual channel $\L^*$, so that it commutes with the Kraus operators $L_i, L_i^*$ for all $i$ (see for instance \cite{CJ}, in particular Proposition 1 applied to the fast recurrent channel $\L$ restricted to $\chi_\alpha$).

Moreover, since $\V$ and ${\mathcal W}$ are minimal, they are the support of two invariant states, that we can denote by $\tau_0^\V$ and $\tau_0^{\mathcal W}$ and will verify
 $$
 \tau_0^\V=Q^* \tau_0^{\mathcal W} Q.
 $$
Then we have
$$
\tr(L_i \tau_0^{\mathcal W} L_i^*)
= \tr(L_i Q\tau_0^{\V} Q^*L_i^*)
= \tr(Q L_i \tau_0^{\V} L_i^* Q^*) 
= \tr(p_{\V} L_i \tau_0^{\V} L_i^*)
= \tr( L_i \tau_0^{\V} L_i^*)
$$
so that
\[
m({\mathcal W}) = \sum_i \tr(L_i \tau_0^{\mathcal W} L_i^*) s_i
= \sum_i \tr( L_i \tau_0^{\V} L_i^*) s_i
= m(\V).
\]

Similarly we deduce, for any $u \in \rr^d$,
\[{\L^\prime_{|\V,u}}(Q^* \cdot Q)=Q^*{\L^\prime_{|{\cal W},u}}(\cdot )Q, \quad {\L^{\prime\prime}_{|\V,u}}(Q^* \cdot Q)=Q^*{\L^{\prime\prime}_{|{\cal W},u}}(\cdot )Q.
\]
Therefore
\[
\tr({\L_{|\V,u}^\prime}(\tau_0^{\V}))
=\tr({\L_{|{\cal W},u}^\prime}(\tau_0^{\mathcal W}))
\qquad \mbox{and} \qquad
\tr({\L_{|\V,u}^{\prime\prime}}(\tau_0^{\V}))=\tr({\L_{|{\cal W},u}^{\prime\prime}}(\tau_0^{\mathcal W})).
\]
By the same arguments, for all $u \in \rr^d$,
\[\eta^\V_u=Q^*\eta^{\cal W}_u Q
\qquad \mbox{and} \qquad
\tr({\L_{|\V,u}^\prime}(\eta^\V_u))=\tr({\L_{|{\cal W},u}^\prime}(\eta^{\cal W}))
\]
and we can conclude that $D(\V)=D({\mathcal W})$.
\end{proof}

{\bf Decomposition of the measure ${\mathbb P}_\rho$.}

In Lemma \ref{lem:supermartingale}, we fixed an enclosure $\V$ and we introduced the probability measure denoted by $\pp'_\rho$. Now we need to handle different enclosures, the ones appearing in the decomposition of $\mathcal R$ given in relations \eqref{R-decomposition}.  We need to highlight the dependence on the enclosure and we shall denote from now on by 
$\mathbb{P}^\alpha$ (resp.ly $\mathbb{P}^{\alpha,\beta}$) the measure $\mathbb{P}'_\rho$ obtained with $\V=\chi_\alpha$ (resp.ly $\V=\V_{\alpha,\beta}$), i.e. with densities
\begin{equation} \label{eq:density2}
\frac{d\pp^\alpha_\rho}{d\pp_\rho} \bigg\rvert_{{\cal F}_n}
=\frac{\tr(A(\chi_\alpha)\rho_n) }{\mathbb{E}_\rho[\tr(A(\chi_\alpha)\rho_0)]},
\qquad\qquad
\frac{d\pp^{\alpha,\beta}_\rho}{d\pp_\rho} \bigg\rvert_{{\cal F}_n}
=\frac{\tr(A(\V_{\alpha,\beta})\rho_n) }{\mathbb{E}_\rho[\tr(A(\V_{\alpha,\beta})\rho_0)]}.
\end{equation}

We can then decompose $\pp_\rho$ into a mixture of $\pp_\rho^\alpha$ and $\pp_\rho^{\alpha,\beta}$.

\begin{lemma} \label{lem:dis}
For any $\alpha \in A$, $\beta \in I_\alpha$ let us define
\begin{eqnarray*}
a_\alpha(\rho)&:=&\mathbb{E}_\rho[Y^\alpha_0]
=\mathbb{E}_\rho[\tr(A(\chi_\alpha)\rho_0)]
=\sum_{\underline{k} \in V} \tr(A(\chi_\alpha) \rho(\underline{k}))
\\
\mbox{and}\qquad
a_{\alpha,\beta}(\rho)&:=&\mathbb{E}_\rho[Y^{\alpha,\beta}_0]=\mathbb{E}_\rho[\tr(A(\V_{\alpha,\beta})\rho_0)]=\sum_{\underline{k} \in V} \tr(A(\V_{\alpha,\beta}) \rho(\underline{k})).
\end{eqnarray*}
Then we can write $\pp_\rho$ as convex combination 
\begin{equation}\label{eq:mixture}
\pp_\rho=\sum_{\alpha \in A} a_\alpha(\rho)\pp_\rho^\alpha=\sum_{\alpha \in A}\sum_{\beta \in I_\alpha} a_{\alpha,\beta}(\rho)\pp_\rho^{\alpha,\beta}.
\end{equation}
\end{lemma}

\begin{proof}
Indeed, for every $\underline{k} \in V$, $n \geq 0$, $\underline{j} \in J^n$
\[\begin{split}
\pp_\rho( \{(\underline{k}, \underline{j})\}\times J^\nn)&=\tr(L_{\underline{j}}\rho(\underline{k})L_{\underline{j}}^*)=\sum_{\alpha \in A}\tr(A(\chi_\alpha)L_{\underline{j}}\rho(\underline{k})L_{\underline{j}}^*)\\
&=\sum_{\alpha \in A} a_\alpha(\rho) \cdot \tr(L_{\underline{j}}\rho(\underline{k})L_{\underline{j}}^*) \frac{1}{\mathbb{E}_\rho[\tr(A(\chi_\alpha)\rho_0)]}\tr \left (A(\chi_\alpha)\frac{L_{\underline{j}}\rho(\underline{k})L_{\underline{j}}^*}{\tr(L_{\underline{j}}\rho(\underline{k})L_{\underline{j}}^*)} \right )\\
&=\sum_{\alpha \in A} a_\alpha(\rho) \pp_\rho^\alpha(\{(\underline{k}, \underline{j})\}\times J^\nn).
\end{split}
\]
where the second equality follows because $\sum_{\alpha \in A}A(\chi_\alpha)=1_\hh$. Similarly one can further decompose the probability measure in $\pp_\rho^{\alpha,\beta}$ because for every $\alpha \in A$,  $\sum_{\beta \in B\alpha}A(\V_{\alpha,\beta})=A(\chi_\alpha)$. Equation \eqref{eq:mixture} is then true because sets of the form $\{(\underline{k}, \underline{j})\}\times J^\mathbb{N}$ generate ${\cal F}$.
\end{proof}

Before proceeding forward, we can now complete the proof of Lemma \ref{lem:supermartingale} and deduce relations \eqref{eq:absdensity}.
\begin{proof} (of Lemma \ref{lem:supermartingale} - second part).\\
First notice the following set equivalence:
\[\left \{\frac{d\pp'_\rho}{d\pp_\rho}=\frac{1}{\mathbb{E}_\rho[Y_0]} \right \}=\left \{ Y_\infty=1\right \}, \quad \left \{\frac{d\pp'_\rho}{d\pp_\rho}=0 \right \}=\left \{Y_\infty=0\right \}.
\]
Let us denote by $q$ the orthogonal projection onto the eigenspace corresponding to the eigenvaule $1$ of $\AV$; since $0 \le \AV \le 1$, $Y_\infty=0$ ($Y_\infty=1$) if and only if $\lim_{n\rightarrow + \infty} \lVert \tilde{p}^\perp_\V \rho_n \tilde{p}^\perp_\V-\rho_n \rVert=0$ ($\lim_{n\rightarrow + \infty} q \rho_n q-\rho_n \rVert=0$). By \cite[Theorem 14]{CG}, we know that $q-p_\V \leq p_\TT$, hence to conclude we only need to show that $\lim_{n\rightarrow + \infty}\lVert  p_\TT \rho_n p_\TT \rVert=0$. Since $p_\TT$ is superharmonic, $T_n:=\tr(p_\TT \rho_n)$ is a supermartingale:
\[\mathbb{E}_{\rho}[T_{n+1}|{\cal F}_n]
=\sum_{i=1}^{v} \tr(L_i \rho_n L_i^*) \frac{ \tr(p_\TT L_i \rho_n L_i^*)}{ \tr(L_i \rho_n L_i^*)}
=\tr(\L^*(p_\TT)\rho_n )\leq \tr(p_\TT \rho_n)=T_n.
\]
Furthermore $0 \le T_n \le 1$, hence $T_n$ converges $\pp_\rho$-a.s. to a certain limit $T_\infty$. Notice that $\mathbb{E}_\rho[T_\infty]= \lim_{n \rightarrow +\infty}\mathbb{E}_\rho[T_n]= \lim_{n \rightarrow +\infty}\L^{*n}(p_\TT)=0$, hence $T_\infty=0$, which implies that $\lim_{n\rightarrow + \infty}\lVert  p_\TT \rho_n p_\TT \rVert=0$.
\end{proof}
{\bf Generalized Central Limit Theorem} 

The convergence in law is metrizable by different distances. On this subject, we refer for instance to \cite{Du}. Among them, we choose the Fortet-Mourier metric, but the convergence results keep holding true also with a different choice.
Let us denote by BL the set of bounded Lipschitz functions on $\rr^d$ equipped with the norm
\[\lVert f \rVert_{BL}= \sup_{x \in \rr^d} |f(x)|+ \sup_{x\neq y} \frac{|f(x)-f(y)|}{| x-y |};
\]
we introduce the Fortet-Mourier distance between two probability laws $P,Q$ on $\rr^d$,
\[{\rm dist}(P,Q):=\sup \left \{\Big\lvert \int_{\rr^d} f dP-\int_{\rr^d} f dQ \Big\rvert : f\in BL, \lVert f \rVert_{BL} \leq 1\right \}.
\]
We recall that \cite[Theorem 11.3.3]{Du}, for $P_n$, $P$ probability measures on $\rr^d$, the following fact holds
$$ P_n  \rightarrow P \mbox{ in law }
\qquad \mbox{if and only if} \qquad
{\rm dist}(P_n,P)\rightarrow 0.$$

We are now in a position to state the ``generalized Central Limit Theorem''.

\begin{theorem}  \label{thm:clt} 
{\bf Convergence to mixture of Gaussians.}\\
Take $m_\alpha$ and $D_{\alpha}$  as in Lemma \ref{lemma-m-D} and let $\pp_{\rho,n}$ be the law of the process $\frac{X_n-X_0}{\sqrt{n}}$ under $\pp_\rho$. Then
\[
\lim_{n \rightarrow +\infty} {\rm dist}\left (\pp_{\rho,n},\sum_{\alpha \in A} a_{\alpha}(\rho) {\mathcal N}(\sqrt{n} m_\alpha,D_{\alpha}) \right )=0,
\]
where $a_{\alpha}(\rho)=\mathbb{E}_\rho[\tr(A(\chi_\alpha)\rho_0)]$ and ${\mathcal N}(\sqrt{n} m_\alpha,D_{\alpha})$ denotes the Gaussian measure with mean $\sqrt{n} m_\alpha$ and covariance matrix $D_{\alpha}$.
\end{theorem}

\begin{proof} 
By Theorem \ref{th:singleCTL}, we know that the process $\frac{X_n-X_0-nm_\alpha}{\sqrt{n}}$ converges in law to a centered normal distribution with covariance matrix $D_\alpha$ under the measure $\pp_\rho^{\alpha,\beta}$, so that we can write
$$
\lim_{n \rightarrow +\infty}{\rm dist} \left (\pp_\rho^{\alpha,\beta} \left (\frac{X_n-X_0-nm_\alpha}{\sqrt{n}}\right ),{\mathcal N}(0,D_\alpha)\right )
=0.
$$
By definition, the Fortet-Mourier distance is invariant with respect to translations and consequently we deduce
$$
{\rm dist}\left (\pp_\rho^{\alpha,\beta} \left (\frac{X_n-X_0}{\sqrt{n}} \right ),{\mathcal N}(\sqrt n m_\alpha,D_\alpha) \right )
\longrightarrow 0, \qquad \mbox{as } n\rightarrow +\infty.
$$
Now, since this limit does not depend on $\beta$ and, by equation \eqref{eq:mixture}
$\pp_\rho^\alpha=\sum_{\beta \in I_\alpha} a_{\alpha,\beta}(\rho)\pp_\rho^{\alpha,\beta}$
(we denote by ${\mathcal N}_\alpha$ the law ${\mathcal N}(\sqrt n m_\alpha,D_\alpha) $ to shorten the expressions in this proof),
\begin{eqnarray*}
&&{\rm dist}\left (\pp_\rho^{\alpha}\left (\frac{X_n-X_0}{\sqrt{n}} \right ),{\mathcal N}_\alpha\right )=\\
&&=\sup \left \{\Big\lvert \int_{\rr^d} f \left (\frac{X_n-X_0}{\sqrt{n}} \right ) d\pp_\rho^{\alpha}
-\int_{\rr^d} f d{\mathcal N}_\alpha
\Big\rvert : f\in BL, \lVert f \rVert_{BL} \leq 1\right \}
 \\
 &&\le \sum_{\beta \in I_\alpha} a_{\alpha,\beta}(\rho)
 \sup \left \{\Big\lvert \int_{\rr^d} f \left (\frac{X_n-X_0}{\sqrt{n}} \right ) d\pp_\rho^{\alpha,\beta}
-\int_{\rr^d} f d{\mathcal N}_\alpha
 \Big\rvert : f\in BL, \lVert f \rVert_{BL} \leq 1\right \}
 \\
 &&= \sum_{\beta \in I_\alpha} a_{\alpha,\beta}(\rho) {\rm dist} \left (\pp_\rho^{\alpha,\beta}\left (\frac{X_n-X_0-nm_\alpha}{\sqrt{n}}\right ),{\mathcal N}(0,D_\alpha)\right )
\longrightarrow 0,
\qquad \mbox{as } n\rightarrow +\infty.
\end{eqnarray*}
Similarly, always by relation \eqref{eq:mixture},
$\pp_\rho=\sum_{\alpha\in A} a_{\alpha}(\rho)\pp_\rho^{\alpha}$ and by triangular inequality for any $f$ in $BL$, we can 
call $\nu_n=\sum_{\alpha \in A} a_{\alpha}(\rho) {\mathcal N}_{\alpha}$ and
write
$$
\Big\lvert \int_{\rr^d} f \left (\frac{X_n-X_0}{\sqrt{n}}\right) d\pp_\rho
-\int_{\rr^d} f d\nu_n
 \Big\rvert 
 \le
 \sum_{\alpha \in A} a_{\alpha}(\rho)  
\Big\lvert \int_{\rr^d} f \left (\frac{X_n-X_0}{\sqrt{n}}\right) d\pp_\rho^{\alpha}
-\int_{\rr^d} f d{\mathcal N}_{\alpha}
 \Big\rvert 
$$
and we then conclude
\[{\rm dist} \left (\pp_{\rho,n},\nu_n \right ) \leq 
 \sum_{\alpha \in A}a_{\alpha}(\rho)  {\rm dist} \left (\pp^{\alpha}_{\rho, n},
{\mathcal N}_{\alpha}\right),
\]
which converges to $0$ as $n \rightarrow +\infty$.
\end{proof}
Notice that, while the weights $a_{\alpha}(\rho)$ depend on the initial state and on the transient part of $\L$, the parameters of the Gaussian measures only depend on the restriction to the fast recurrent part.
Theorem \ref{thm:clt} has the following direct consequence on the convergence of the empirical means.

\begin{corollary} \label{coro:deltas}
Let $\hat{\pp}_{\rho,n}$ the law of the process $\frac{X_n-X_0}{n}$ under $\pp_\rho$, then
\[\lim_{n \rightarrow +\infty} {\rm dist} \left (\hat{\pp}_{\rho,n},\sum_{\alpha \in A}a_\alpha(\rho) \delta_{m_\alpha} \right )=0,
\]
where $a_\alpha(\rho)$ are defined as in previous theorem and $\delta_{m_\alpha} $ denotes the Dirac measure concentrated in $m_\alpha$.
\end{corollary}

\begin{remark} \label{rmk:goqrw}
{\bf Possible Extensions}.
As for previous versions of central limit theorems for HOQRWs, we can extend our results to more general cases.

1. There is an immediate generalization of HOQRW obtained considering a change in the local state after a shift $s_i$ given by a quantum operation $\L_j$ with more than one Kraus operator, which is the case we considered ($\L_j(\cdot)=L_j \cdot L_j^*$). In this case it suffices to change the notation in the proof of Theorem \ref{thm:clt} to see that it still holds true.

2. Open quantum random walks have been defined also in continuous time (\cite{Pe}) and the central limit theorem for the position process has already been proved in \cite{BrH}, under the assumption of irreducibility of $\RR$. Theorem \ref{thm:clt} can be carried with some technical adaptations to the continuous time case.
\end{remark}

\begin{remark}\label{comparison}
\textbf{Comparison with previous results.} 
The first CLT for HOQRWs appeared in \cite{AGS} where the authors proved it by the use of Poisson equation and martingale techniques in the case $\RR$ irreducible.  Indeed, in \cite[Theorem 7.3]{AGS} they showed the convergence to different Gaussian measures under proper conditional probabilities and under assumptions which can be translated in our language to be
\\
- $\TT=\{0\}$,\\
- $\chi_\alpha$ is minimal for every $\alpha \in A$,\\
- $m_\alpha \neq m_{\alpha^\prime}$ if $\alpha \neq \alpha^\prime$.

These techniques revealed to be successful to treat also other walks and in particular have recently been exploited also in \cite{KPS} to obtain a CLT for the so-called lazy OQWs.
Successively, in \cite{CP}, an alternative proof of the central limit theorem for an irreducible fast recurrent local channel $\L$ could be deduced from a large deviation principle, proved by deformation techniques. Finally the results in \cite{KoYo, KKSY,KonYo} (which are formulated in the setting of homogeneous open quantum walks on crystal lattices) state a kind of convergence to a mixture of Gaussian measures, under some conditions, always essentially implying that the local channel is fast recurrent. 

Here, with Theorem \ref{thm:clt}, we can find an improvement of all these previous results since we can drop any condition about recurrence or transience or reducibility of the local channel and we can specify the form of convergence to the mixture of Gaussians introducing a metric on the set of probability measures.
Moreover we can specify the weights of the limit mixture in terms of the initial state and of the decomposition of the local space.

We refer the reader to \cite{PS2} for other hints on the existing literature until 2019 and to \cite{BrH,Pe,Pas} for CLT results for different families of open walks.
\end{remark}

\section{Large Deviations} \label{sec:LDP}

When the central limit theorem is approached by Bryc's theorem, it is often treated together with large deviations, and this was indeed the idea in \cite{CP}, where the proof of the central limit theorem in the particular case of an irreducible fast recurrent subspace was a byproduct of the large deviation principle.
Similarly, it is here natural to wonder whether a large deviation principle can hold in general for the position process of a HOQRW, always under the measure $\pp_\rho$ induced by the initial state $\rho$.
We shall prove that G\"artner-Ellis' theorem can be applied and thus large deviations hold when the local map is recurrent. Moreover, the rate function is related to the spectrum of the deformed map $\L_u$. When instead there is a non-trivial transient subspace for the local channel $\L$, the limit of the moment generating functions is not smooth in general, as \cite[Example 7.3]{CP} shows, and G\"artner-Ellis' theorem will simply provide lower and upper bounds.

As for the results in the previous section, only the minimal enclosures in the decomposition of $\R$ that are ``reachable'' by a initial state $\rho$ will play a role in the large deviations results. 
For this reason, it is useful to remember the definition of the quantities 
$a_\alpha(\rho)$, $a_{\alpha,\beta}(\rho)$ (introduced in Lemma \ref{lem:dis}), which are a kind of quantum absorption probabilities of the evolution in the enclosures $\chi_\alpha$, or $\V_{\alpha,\beta}$ respectively, when the initial state is $\rho$.
Differently from the central limit type results, here also the index $\beta$, and so the particular enclosures $\V_{\alpha,\beta}$ selected inside $\chi_\alpha$ are important, and this is related to the fact that the evolution on the transient subspace affects large deviations results.

Since we need to define restrictions of the channel $\L$ which take into account only proper subspaces reachable by the local initial states $\rho(\underline{k})$, we define the subspace
$$
{\cal E}(\rho):= {\rm span} \{{\rm supp}(\L^n(\rho(\underline{k}))), \underline{k}\in V, n\ge 0\}
\subset \hh,
$$
which is an enclosure due to \cite[Propositions 4.1 and 4.2]{PC}.

We recall that by $\hat{\pp}_{\rho,n}$ we denote the law of $\frac{X_n-X_0}{n}$ under $\pp_\rho$ and, given any enclosure $\V$, $\tilde{p}_\V$ is the orthogonal projection onto ${\rm supp}(\AV)$.

\begin{theorem} {\bf Large deviation principle.} \label{thm:ld}
Suppose that the local map $\L$ is recurrent, i.e. $\R=\hh$. Then  $(\hat{\pp}_{\rho,n})_{n\ge 1}$ satisfies a large deviation principle with good rate function
\[\Lambda_\rho(x)=\min_{\alpha \,: \, a_\alpha(\rho)\neq 0}\Lambda_\alpha(x),\]
where $\Lambda_\alpha$ is the Fenchel-Legendre transform of the logarithm of the spectral radius $\lambda_{\alpha,u}$ of $\L_{|\chi_\alpha,u}$, i.e.
$$
\lambda_{\alpha,u} = r(\L_{|\chi_\alpha,u}), \qquad
\Lambda_\alpha(x)=\sup_{u \in \mathbb{R}^d}\{ \langle u,x \rangle -\log(\lambda_{\alpha,u})\}
\quad x\in \mathbb{R}^d.
$$
\end{theorem}

\begin{theorem} {\bf Large deviations upper and lower bounds.} 
\label{thm:ld1}
For any measurable $B \in \mathcal{B}(\mathbb{R}^d)$
\begin{itemize}
\item $\limsup_{n \rightarrow +\infty} \frac{1}{n} \log(\hat{\pp}_{\rho,n}(B)) 
\leq - \inf_{x \in \overline{B}} \min_{\alpha ,\beta \,:\, a_{\alpha,\beta}(\rho)\neq 0}\Lambda_{\alpha,\beta}^\rho(x)$,
\item $\liminf_{n \rightarrow +\infty} \frac{1}{n} \log(\hat{\pp}_{\rho,n}(B)) 
\geq  -\min_{\alpha ,\beta \,:\, a_{\alpha,\beta}(\rho)\neq 0}
 \inf_{x \in \mathring{B} \cap {{\cal S}_{\alpha,\beta}}}\Lambda_{\alpha,\beta}^\rho(x)$
\end{itemize}
where
\begin{itemize}
\item $\lambda^\rho_{\alpha,\beta,u}=r(\L_{|{\cal Q}_{\alpha,\beta}^\rho,u})$ for ${\cal Q}^\rho_{\alpha,\beta}:=\tilde{p}_{\V_{\alpha,\beta}} {\cal E}(\rho)$,
\item $\Lambda^\rho_{\alpha,\beta}(x)=\sup_{u \in \mathbb{R}^d}\{ \langle u,x \rangle -\log(\lambda^\rho_{\alpha,\beta,u})\}$
is the Fenchel-Legendre transform of $\log(\lambda^\rho_{\alpha,\beta,u})$,
\item ${\cal S}_{\alpha,\beta}=\rr^d$ if $\lambda^\rho_{\alpha,\beta,u}$ is smooth, otherwise ${\cal S}_{\alpha,\beta}$ is the set of exposed points of $\Lambda^\rho_{\alpha,\beta}$ (see \cite[Definition 2.3.3]{dembo2011large}).
\end{itemize}
\end{theorem}

\begin{remark}
We remark that, whenever $a_{\alpha,\beta}(\rho)\neq 0$, ${\cal Q}^\rho_{\alpha,\beta}$ is non trivial and
$$
{\cal Q}^\rho_{\alpha,\beta}=\V_{\alpha,\beta} \oplus (\TT \cap {\cal Q}^\rho_{\alpha,\beta})\subset {\rm supp}(A(\V_{\alpha,\beta}))
$$
(see the first step in the proof of Theorem \ref{thm:ld1}).
\end{remark}

We shall prove the two theorems in inverse order. The proof will request different steps and we shall proceed similarly as we did for central limit theorems, first considering the measure $\pp'_\rho$ associated with the absorption in a single minimal enclosure (Lemma \ref{lem:supermartingale}), and then generalizing using the expression of $\pp_\rho$ as a convex combination given in Lemma \ref{lem:dis}.

\begin{proof}[Proof of Theorem \ref{thm:ld1}]

{\bf Step 1.}
We fix the initial state $\rho$ and a minimal enclosure $\V$, whose corresponding absorption operator is denoted as usual by $\AV$. If $\mathbb{E}_\rho[\tr(\AV\rho_0)]>0$, we introduce the measure $\pp'_\rho$ as previously in Lemma \ref{lem:supermartingale}. This first step consists in proving large deviations bounds for the position process under the measure $\pp'_\rho$.

We need to consider a restriction of the channel $\L$ which takes into account only the subspace of ${\rm supp}\AV$ which is someway reachable by the local initial states $\rho(\underline{k})$. To this aim we use the enclosure ${\cal E}(\rho)$ and define the subspace 
$$
{\cal Q}=\tilde{p}_\V {\cal E}(\rho).
$$

\begin{enumerate}

\item 
${\cal Q} \oplus ({\cal E}(\rho)^\perp \cap {\rm supp}(\AV) ) = {\rm supp}(\AV)$.

Indeed, $v\in {\cal Q}^\perp \cap {\rm supp}(\AV)$ if and only if
$$ v \in {\rm supp}(\AV) \text{ and, } 
 \forall w\in {\cal E}(\rho), \; 
 0=\langle v, \tilde{p}_\V(w)\rangle = \langle \tilde{p}_\V(v), w \rangle =
\langle v, w \rangle,
$$
i.e. $v \in {\rm supp}(\AV) \cap {\cal E}(\rho)^\perp$.

\item $\mathbb{E}_\rho[\tr(\AV\rho_0)]=0$ if and only if ${\cal Q}=\{0\}$.

Since $\tr(\AV\rho_0)$ is a non negative random variable, it has zero mean if and only if it is almost surely null, that is
\begin{eqnarray*}
&\Leftrightarrow&
0= \tr(A(\V)\rho({\underline k}))=  \tr(\AV \L^n(\rho({\underline k}))) \quad \forall {\underline k}\in V\\
\mbox{(since $\L(\AV)=\AV$)} &\Leftrightarrow&
\tr(\AV \L^n(\rho({\underline k})))= 0 \quad \forall {\underline k}\in V,n\ge 0 \\
&\Leftrightarrow& {\tilde p}_\V({\rm supp} (\L^n(\rho({\underline k})))=\{0\} \quad \forall {\underline k},n 
\end{eqnarray*}
 
\item Otherwise $\mathbb{E}_\rho[\tr(\AV\rho_0)]>0$ and $\V \subset {\cal Q}$.

By using the same ideas as before,
if $\mathbb{E}_\rho[\tr(\AV\rho_0)]>0$, ${\cal Q}$ is non trivial and there exist some ${\underline k}\in V,n\ge 0$ such that 
$\tr(p_\V \L^n(\rho({\underline k}))) \neq 0$ and this implies
$$
\{0\}\neq p_\V({\cal E}(\rho)) = p_\V(p_{\cal R}({\cal E}(\rho)))
= p_\V(\R \cap {\cal E}(\rho))
$$
where the last equality follows from \cite[Proposition 23]{CG}.
So $(\R \cap {\cal E}(\rho))$ is a non null positive recurrent enclosure (as intersection of enclosures) and it is non orthogonal to $\V$, hence it contains a minimal enclosure ${\cal W}$ which is in the same $\chi_\alpha$ as $\V$ and is not orthogonal to $\V$. Then, by using the partial isometry $Q$ as in relation \eqref{op Q}, we deduce that 
$$
\V =p_\V({\cal W}) \subset {\tilde p}_\V ({\cal E}(\rho)) ={\cal Q}.
$$
\end{enumerate}

We call $\Phi$ the restriction of $\L$ to the subspace ${\cal Q}$, $\Phi(\sigma)=p_{\cal Q}\L(p_{\cal Q}\sigma p_{\cal Q})p_{\cal Q}$ and $\Phi_u$ its deformation.
${\cal Q}$ and consequently $\Phi$ obviously depend on the enclosure $\V$ and on the initial state $\rho$, but we do not need to highlight this in the notations. 

\begin{lemma} \label{lem:singleV}
Suppose $\mathbb{E}_\rho[\tr(\AV\rho_0)]>0$.
For any measurable $B \in \mathcal{B}(\mathbb{R}^d)$
\begin{itemize}
\item $\limsup_{n \rightarrow +\infty} \frac{1}{n} \log(\hat{\pp}^\prime_{\rho,n}(B)) \leq - \inf_{x \in \overline{B}}\Lambda(x)$;
\item $\liminf_{n \rightarrow +\infty} \frac{1}{n} \log(\hat{\pp}^\prime_{\rho,n}(B)) \geq  - \inf_{x \in \mathring{B} \cap {\cal S}}\Lambda(x)$
\end{itemize}
where
\begin{itemize}
\item $\Lambda$ is the Fenchel-Legendre transform of $\log(\lambda^\rho_u)$, 
\item $\lambda^\rho_u$ is the spectral radius of $\Phi_u$ ,
\item ${\cal S}=\rr^d$ if $\lambda^\rho_u$ is smooth, otherwise it corresponds to the set of exposed points of $\Lambda$.
\end{itemize}
\end{lemma}

\begin{proof}
In order to apply \cite[Theorem 2.3.6]{dembo2011large}, we need to prove that for every $u \in \rr^d$ we have 
\[
\lim_{n \rightarrow +\infty} \frac{1}{n}\log(\mathbb{E}^\prime _\rho[{\rm e}^{u \cdot X_n-X_0 )}])=\log(\lambda^\rho
_u).
\]
Notice that we computed the same limit in the proof of Theorem \ref{th:singleCTL}, but for $u$ in a complex neighborhood of the origin.

For any $n \in \mathbb{N}$, by construction $\Phi^n_u(\rho(\underline{k}))=\tilde\L^n_u(\rho(\underline{k}))$ for all $\underline{k}$ and $u$, so we can write

\[\begin{split}
\mathbb{E}_\rho[\tr(\AV \rho_0] \cdot \mathbb{E}^\prime_\rho[{\rm e}^{u \cdot (X_n-X_0)}]
&=\sum_{\underline{k} \in V} \tr(\AV\tilde{\L}_u^n({\rho}(\underline{k})))
=\sum_{\underline{k} \in V} \tr(\AV\Phi_u^n({\rho}(\underline{k})))\\
&=\sum_{\underline{k} \in V} \tr({\rho}(\underline{k}) \Phi_u^{*n}(\AV))\\
&\leq \|\sum_{\underline{k} \in V} {\rho}(\underline{k})\|_{L^1} \, \| \Phi_u^{*n}({\AV}))\|_\infty 
\leq  \lVert \Phi_u^{*n} \rVert_\infty.
\end{split}
\]
Because of Gelfand formula, we get
\[
\limsup_{n \rightarrow +\infty} \frac{1}{n}\log(\mathbb{E}^\prime_\rho[{\rm e}^{u \cdot (X_n-X_0 )}]) 
\leq \log\left (\lim_{n \rightarrow +\infty} \lVert \Phi_u^{*n} \rVert^{1/n}_\infty \right)
=\log(\lambda^\rho_u).
\]
Now consider $w_{u} \in B(\hh)$ the Perron-Frobenius eigenvector for $\Phi^*_u$, i.e. such that 
$\Phi^*_u(w_u)=\lambda^\rho_u w_u$. $w_u$ is a non null positive operator supported in ${\cal Q}$, so
there exist $N \in \nn$ and $\hat{\underline{k}}$ in $V$ such that 
$\tr(\tilde{\L}^N({\rho}(\hat{\underline{k}}))w_u) \neq 0$. Therefore $\tr(\Phi^N_{u}({\rho}(\hat{\underline{k}}))w_u)=\tr(\tilde{\L}^N_{u}({\rho}(\hat{\underline{k}}))w_u)\neq 0$. 

Since ${\cal Q}$ is finite dimensional, there exists a constant $M>0$ such that $p_{\cal Q}\AV p_{\cal Q} \geq M w_u$, hence for every $n \geq N$ we have
\[\begin{split}
\mathbb{E}_\rho[\tr(\AV \rho_0] \cdot \mathbb{E}^\prime_\rho[{\rm e}^{u \cdot (X_n-X_0)}] 
&= \sum_{\underline{k} \in V} \tr\left(\AV  \tilde{\L}_{u}^n({\rho}(\underline{k}))\right)\\
&\geq \tr\left(\AV  \tilde{\L}_{u}^n({\rho}(\hat{\underline{k}}))\right)\\
&= \tr\left(\AV  \Phi_{u}^n({\rho}(\hat{\underline{k}}))\right)\\
&\geq  M \tr\left(\Phi_{u}^N({\rho}(\hat{\underline{k}})) \Phi_{u}^{*(n-N)}(w_u)\right)\\
& = M\tr\left(\Phi_{u}^N({\rho}(\hat{\underline{k}}))w_u\right) (\lambda^\rho_{u})^{n-N}.
\end{split}\]
Therefore
\[\liminf_{n \rightarrow +\infty} \frac{1}{n}\log(\mathbb{E}^\prime_\rho[{\rm e}^{u \cdot (X_n-X_0)}]) \geq \log(\lambda^\rho_{u}).\]
This allows to compute the desired limit and the statement follows by direct application of the G\"artner-Ellis' theorem. Notice that we do not have to worry about the domain of $\log(\lambda_u^\rho)$ since it is easy to see that $\lambda_u^\rho$ is a strictly positive real number for every $u \in \rr^d$.
\end{proof}

{\bf Step 2.} 
We complete the proof of the statement of the theorem by using the expression of $\pp_\rho$ as convex combinations of the $\pp_\rho^{\alpha,\beta}$ deduced in relation \eqref{eq:mixture}.
This implies that a similar decomposition holds for $\hat{\pp}_{\rho,n}$ in terms of $(\hat{\pp}_{\rho,n}^{\alpha,\beta})_{\alpha,\beta}$, i.e.
\begin{equation*}
\hat{\pp}_{\rho,n}
=\sum_{\alpha \in A}\sum_{\beta \in I_\alpha} a_{\alpha,\beta}(\rho)\hat{\pp}_{\rho,n}^{\alpha,\beta}.
=\sum_{j\in J_{\rho}} a_{j}(\rho) \hat{\pp}_{\rho,n}^{j},
\end{equation*}
$$
\mbox{where}\qquad
J_\rho:=\left\{(\alpha,\beta): \alpha\in A, \beta\in I_\alpha:  a_{\alpha,\beta}(\rho)>0\right\}.
$$

Since, for any $j\in J_\rho$ and $B \in \mathcal{B}(\rr^d)$, $\hat{\pp}_{\rho,n}(B) \geq a_\alpha \hat{\pp}_{\rho,n}^{j}(B)$, we trivially have
\[\begin{split}
&\liminf_{n \rightarrow +\infty} \frac{1}{n} \log(\hat{\pp}_{\rho,n}(B))
\geq \max_{j\in J_\rho} \liminf_{n \rightarrow +\infty} \frac{1}{n} \log(\hat{\pp}_{\rho,n}^{j}(B)), 
\\
&\limsup_{n \rightarrow +\infty} \frac{1}{n} \log(\hat{\pp}_{\rho,n}(B))
\geq \max_{j\in J_\rho} \limsup_{n \rightarrow +\infty} \frac{1}{n} \log(\hat{\pp}_{\rho,n}^{j}(B)).
\\
\end{split}
\]
Then we have
\[\begin{split}
\limsup_{n \rightarrow +\infty} \frac{1}{n} \log(\hat{\pp}_{\rho,n}(B))
&\leq \underbrace{\limsup_{n \rightarrow +\infty} \frac{1}{n} \log(|J_\rho|)}_{=0}+ \limsup_{n \rightarrow +\infty} \frac{1}{n} \log \left (\max_{j\in J_\rho}\hat{\pp}_{\rho,n}^{j}(B)\right )\\
& =\max_{j\in J_\rho}\limsup_{n \rightarrow +\infty} \frac{1}{n} \log(\hat{\pp}_{\rho,n}^{j}(B))
\end{split}\]
and we are done.
\end{proof} 

\begin{proof} [Proof of Theorem \ref{thm:ld}]
Under the hypothesis $\hh=\R$, we have that $A(\V_{\alpha,\beta})=p_{\V_{\alpha,\beta}}$, which implies ${\cal Q}_{\alpha, \beta}=\V_{\alpha,\beta}$ and $\tilde{\L}_{\alpha,\beta,u}=\L_{|\V_{\alpha,\beta},u}$.

Since $\L_{|\V_{\alpha,\beta}}$ is irreducible, $\lambda_{\alpha,\beta,u}$ is an analytic function of $u \in \rr^d$ (\cite[Lemma 5.3]{CP}) and consequently $S_{\alpha,\beta}=\rr^d$.

Moreover recall (equation \eqref{psi}) that $\L^*_{|\chi_\alpha}$ is unitarily equivalent to ${\rm Id}_{B(\cc^{|I_\alpha|})} \otimes \psi$ where $\psi$ is equal to $\L^*_{|\V_{\alpha,\beta}}$, hence $\L_{|\chi_\alpha}$ and $\L_{|\V_{\alpha,\beta}}$ have the same spectral radius.

Therefore the following equality holds:
\[\min_{(\alpha,\beta) \in J_\rho}\Lambda_{\alpha,\beta}=\min_{\alpha: a_{\alpha}(\rho)\neq 0}\Lambda_\alpha.
\]
Theorem \ref{thm:ld1} ensures that  for any measurable $B \in \mathcal{B}(\mathbb{R}^d)$
\begin{itemize}
\item $\limsup_{n \rightarrow +\infty} \frac{1}{n} \log(\hat{\pp}_{\rho,n}(B)) \leq - \inf_{x \in \overline{B}} \min_{\alpha: a_{\alpha}(\rho)\neq 0}\Lambda_{\alpha}(x)$,
\item $\liminf_{n \rightarrow +\infty} \frac{1}{n} \log(\hat{\pp}_{\rho,n}(B)) \geq  - \inf_{x \in \mathring{B} } \min_{\alpha: a_{\alpha}(\rho)\neq 0}\Lambda_{\alpha}(x),$
\end{itemize}
which is exactly the definition of large deviation principle with rate function $\Lambda_\rho(x):=\min_{\alpha: a_{\alpha}(\rho)\neq 0}\Lambda_{\alpha}(x)$, $x\in \erre^{d}$. Note that $\Lambda_\rho$ has compact level sets because every $\Lambda_\alpha$ does (it is a consequence of G\"artner-Ellis' theorem).
\end{proof}

Consider a minimal enclosure $\V$ such that $\mathbb{E}_\rho[\tr(\AV \rho_0)] \neq 0$; taking the notations of the first step in the proof of Theorem \ref{thm:ld1}, the following proposition states that $\lambda^\rho_u$ can be seen as the result of two contributions: one depending on the recurrent dynamic on $\V$ and the other one on the transient dynamic on its orthogonal complement in ${\cal Q}$, which we denote by ${\cal W}:={\cal Q}\cap \TT$.

\begin{proposition}
\label{lambdaRT}
Let $\lambda_u^\V$ and $\lambda_u^{\cal W}$ be the spectral radii of $\Phi_{|\V,u}$ and $\Phi^*_{|{\cal W} ,u}$ respectively. Then $\lambda_u^\rho=r(\tilde{\L}_u)=\max \{\lambda_u^\V,\lambda_u^{\cal W}\}$.
\end{proposition}

\begin{proof}
We only need to prove that if $\lambda_u^{\rho} > \lambda_u^\V$, then $\lambda_u^{\rho} = \lambda_u^{\cal W}$. Theorem \ref{pf0} tells us that there exists a positive $\omega_u \in L^1({\cal Q})$ such that $\Phi_u(\omega_u)=\lambda^\rho_u \omega_u$; since $\lambda^\rho_u > \lambda_u^\V$, it must be true that $p_{\cal W} \omega_u p_{\cal W} \neq 0$ and we have the following:
\[p_{\cal W} \Phi_u(p_{\cal W} \omega_u p_{\cal W})p_{\cal W}=p_{\cal W} \Phi_u( \omega_u  )p_{\cal W}=\lambda_u p_{\cal W} \omega_u p_{\cal W}.
\]
The first equality follows from the fact that for any $\rho \in L^1(\hh)$
\[
\Phi_u(p_\V \rho )=p_{\cal Q} \L_u ( p_\V \rho p_{\cal Q} )p_{\cal Q}\overset{\text{ ($\V$ is an enclosure)}}{=}p_\V \L_u ( p_\V \rho p_{\cal Q} )p_{\cal Q} =p_\V \Phi_u(p_\V \rho)
\]
and analogously $\Phi_u( \rho p_\V)=\Phi(\rho p_\V)p_\V$.

\end{proof}

\section{Examples and numerical simulations} 
\label{sec:examples}

\subsection{Commuting normal local Kraus operators} \label{sub:comm}
As a first family of examples, we consider some HOQRWs studied in \cite{PS}: take $V=\zz^d$ and a local channel with normal commuting Kraus operators $\{L_j\}_{j=1}^{2d}$. In this case, there exists an orthonormal basis $\{\phi_i\}_{i=1}^h$ that simultaneously diagonalizes the Kraus operators and we can write $L_j=\sum_{i=1}^h \zeta_{i,j} \ket{\phi_i}\bra{\phi_i}$. 
The normalization condition for the operators $L_j$ given by equation (\ref{oqw1}) implies that $\sum_{j=1}^{2d} |\zeta_{i,j}|^2=1$ for any $i=1,\dots, h$. 

It is easy to verify by direct computation that, for every $i=1,\dots, h$, $\omega_i=\ket{\phi_i}\bra{\phi_i}$ is a minimal invariant state for $\L$, and consequently $\V_i:={\rm span}\{\phi_i\}$ is a minimal recurrent enclosure. 
Hence $\L$ is positive recurrent and $\hh=\oplus_i\V_i$ is a decomposition of the local space $\hh$ in minimal orthogonal enclosures.

However, for our study, we are interested in a decomposition of the form described in \eqref{R-decomposition} and in particular we should identify the enclosures $\chi_\alpha$, which will be given by the direct sum of some of the $\V_i$'; indeed, we can see that
$\V_i$ and $\V_l$ are in the same $\chi_\alpha$ if and only if for every $j=1,\dots, 2d$, $\zeta_{i,j}=\zeta_{l,j}=:\zeta_{\alpha,j}$. This reflects on the structure of the Kraus operators, that will also be written as $L_j=\sum_{\alpha \in A} \zeta_{\alpha,j} p_{\chi_\alpha}$, $j=1,\dots, 2d$.

In this simple example, the probability law of the shift $X_n-X_0$ is a convex combination of $|A|$ multinomial distributions with parameters $(|\zeta_{\alpha,1}|^2,\dots, |\zeta_{\alpha,2d}|^2)$: for every $n \geq 1$
\[\pp_\rho(X_1-X_0=e_{j_1},\dots, X_n-X_{n-1}=e_{j_n})=\sum_{\alpha=1}^{|A|} \underbrace{\sum_{\underline{k} \in \mathbb{Z}^d}\tr(p_{\chi_\alpha} \rho(\underline{k}))}_{=:a_\alpha (\rho)} \prod_{k=1}^n |\zeta_{\alpha,j_k}|^2,
\]
where $e_1,\dots, e_d$ is the canonical basis of $\rr^d$ and $e_{2j}=-e_j$ for $j=1,\dots,d$. Applying the central limit theorem for the mean of i.i.d. random variables, we see that

\begin{equation} \label{eq:mixPS}
\lim_{n \rightarrow +\infty} {\rm dist} \left (\pp_{\rho,n},\sum_{\alpha=1}^{|A|} a_\alpha (\rho) {\mathcal N}\left(\sqrt{n}m_\alpha,D_\alpha\right) \right )=0
\end{equation}
where $m_\alpha=\sum_{j=1}^{2d} |\zeta_{\alpha,j}|^2 e_j$ and $D_\alpha=\sum_{j=1}^d (|\zeta_{\alpha,j}|^2+|\zeta_{\alpha,2j}|^2) \ket{e_j}\bra{e_j}$.  

Similarly, if we apply Theorem \ref{thm:clt}, we find again relation (\ref{eq:mixPS}) (in this case computations for the asymptotic means and covariance matrices are very easy).  

Also, by applying Theorem \ref{thm:ld}, we can state that a large deviations' principle holds for the process $\frac{X_{n}-X_{0}}{n}$ and the rate function is given by 
\[
\Lambda_\rho(x):=\min_{\alpha : a_\alpha(\rho) \neq 0} \Lambda_\alpha(x), \quad x \in \mathbb{R}^d
\]
where $\Lambda_\alpha(x)=\sup_{u \in \mathbb{R}^d} \{\langle u,x \rangle -\log(\lambda_{\alpha,u})\}$ and $\lambda_{\alpha,u}=\sum_{j=1}^{2d}\vert \zeta_{\alpha,j}\vert^{2}{\rm e}^{u\cdot e_{j}}$.
\subsection{An example with non trivial transient space}
\label{Example2}
\begin{figure}[!h] 
\centering
\begin{subfigure}[b]{0.8\textwidth}
\includegraphics[width=0.8\textwidth]{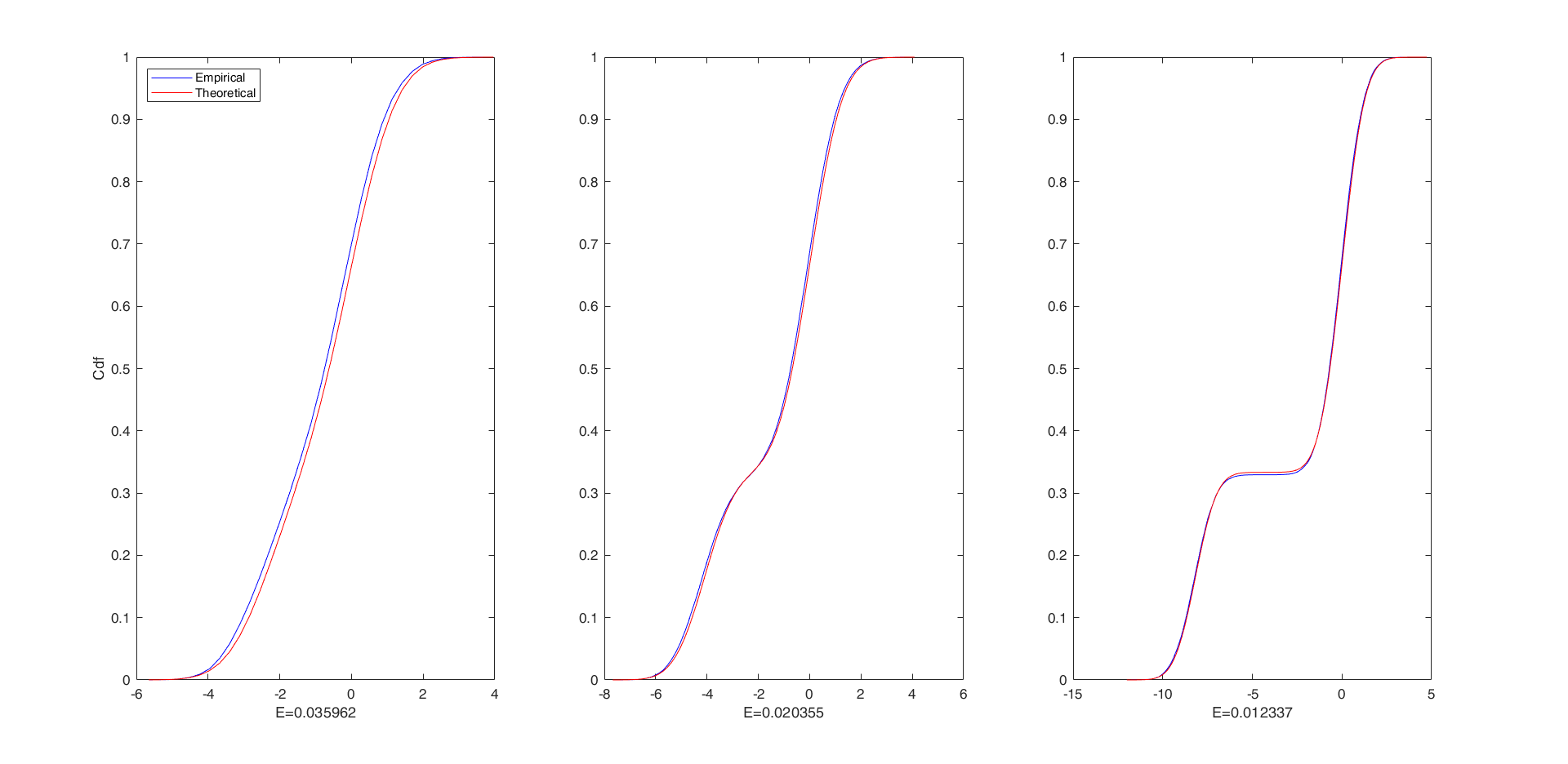}
\caption{$\rho=\frac{1}{3} (\ketbra{e_1}{e_1}+\ketbra{e_2}{e_2}+\ketbra{e_3}{e_3})$ and $p_3=\frac{1}{2}$.}  \label{fig:1}
\end{subfigure}
\begin{subfigure}[b]{0.8\textwidth}
         \includegraphics[width=0.8\textwidth]{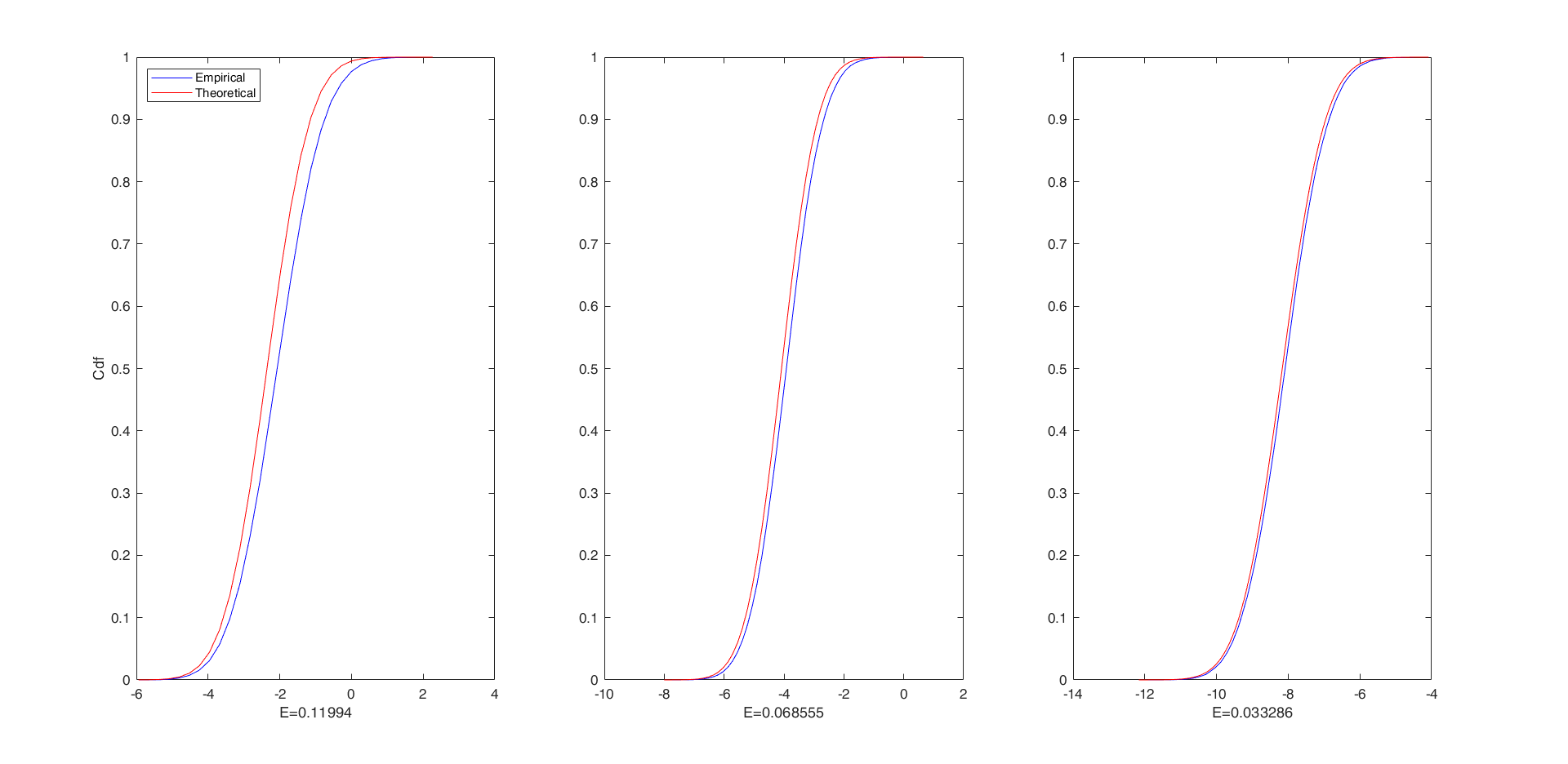} 
        \caption{$\rho=\ket{e_0}\bra{e_0}$ and   $p_3=\frac{1}{2}$.}  \label{fig:2}
     \end{subfigure}
\end{figure} \label{fig:cdfs}
We consider a family of HOQRW with local Hilbert space $\hh={\mathbb C}^4$, including the walk defined in Example \ref{example2}.
We introduce the parameters $p_1,p_2,p_3\geq 0$ such that $\sum_{i=1}^3p_i=\frac{1}{2}$ 
and define left and right Kraus operators
\[
L=\begin{pmatrix}\frac{1}{2\sqrt{2}}& 0&0&0\\[6pt]
\sqrt{\frac{p_1}{2}}& \frac{1}{\sqrt{2}}&0&0\\[6pt]
\sqrt{\frac{p_2}{2}}&0&  \frac{1}{\sqrt{2}}&0\\[6pt] 
-\sqrt{\frac{p_3}{3}}&0& 0& \frac{2}{\sqrt{3}}\\[6pt]   
\end{pmatrix}, 
\quad
R=\begin{pmatrix}\sqrt{\frac{3}{8}}& 0&0&0\\[6pt]
-\sqrt{\frac{p_1}{2}}& \frac{1}{\sqrt{2}}&0&0\\[6pt]
-\sqrt{\frac{p_2}{2}}&0&  \frac{1}{\sqrt{2}}&0\\[6pt] 
\sqrt{\frac{2p_3}{3}}&0& 0& \frac{1}{\sqrt{3}}\\[6pt]   \end{pmatrix}.
\]
Notice that Example 	\ref{example2} corresponds to the case $p_1=p_2=0,\; p_3=1/2$.  

This family of local channels revealed to be very useful since, though with a low dimensional local Hilbert space, it can display already a more sophisticated structure of the decomposition of the local space. Indeed, the transient subspace is non trivial and the recurrent subspace is reducible as a sum of two $\chi_\alpha$, one which is a minimal enclosure and one which is not.

Let $\{e_i\}_{i=0}^3$  be the canonical basis of $\h$. It is immediate to see, for instance by computing explicitly the invariant states of the corresponding local channel $\L$, that ${\cal T}={\rm span}\{e_0\}$, ${\cal R}={\rm span}\{e_1,e_2,e_3\}$ and the decomposition of the recurrent space is the following:
\[
{\cal R}= \underbrace{{\rm span}\{e_1,e_2\}}_{\chi_1}\oplus \underbrace{{\rm span}\{e_3\}}_{\chi_2}.
\]
With simple direct computations one can find the parameters of the limit Gaussians: for the enclosure $\chi_1$ one has mean $m_1=0$ and variance $D_1=1$, while for $\chi_2$ one has parameters $m_2=-\frac{1}{3}$ and $D_2=\frac{8}{9}$. 

For this walk, depending on the different choice of the initial state $\rho$, we can observe either only one of the two Gaussians or various mixtures of the two Gaussians. When the $\rho({\underline k})$'s are all contained in a same $\chi_\alpha$, then we shall see only the Gaussian associated with the same $\chi_\alpha$, $\alpha=1,2$.

In order to consider the asymptotic behavior, we need the following absorption operators:
\[A(\chi_2)=2p_3 \ket{e_0}\bra{e_0} + \ket{e_3}\bra{e_3}, 
\quad 
A(\chi_1)=1_\hh-A(\chi_2).
\]

We can take for simplicity $X_0=0$ and it will be particularly interesting to consider an initial state $\rho$ supported in the transient subspace, and so of the form $\rho =\rho_0\otimes \ket{0}\bra{0}$, with $\rho_0=(\rho_0(i,j))_{i,j=0,...3}$ a non negative unit-trace matrix in $M_4({\mathbb C})$.
Then we can explicitly compute the weights of the Gaussian mixture appearing in the generalized CLT, which will be given by the quantum absorption probabilities
\[\begin{split}
a_1(\rho)&=2p_3 \rho_0(0,0)+\rho_0(3,3),
\\
a_2(\rho)&=1-a_1(\rho)=2(p_1+p_2)\rho_0(0,0) + \rho_0(1,1)+\rho_0(2,2).
\end{split}
\]
We illustrate our result also by numerical simulations. We used $N=5\times 10^4$ samples of $\frac{X_n}{\sqrt{n}}$ for $n =50,150,600$ 
in order to estimate their probability distribution and we compared it with the expected convex combination of Gaussian measures. Figures \ref{fig:ex1} and \ref{fig:ex2} show the histograms of $\frac{X_n-X_0}{\sqrt{n}}$ at the three different times (n=50,150,600) for the choice $p_3=\frac{1}{2}$ and for two different choices of the local initial state $\rho_0$. In Figure \ref{fig:cdfs} we reported the empirical and the expected cumulative function. The same plots for the choice $p_3=\frac{1}{6}$ are reported in Figure \ref{fig:3}. Once again we remark that, tuning initial state and absorption rates the Gaussian laws in the mixture do not change, but only their weights.

Finally, numerical simulations can also help us to have a better intuition of the behavior of the processes $(Y_n)_n$ used to introduce the laws of the family $\pp'_\rho$ (recall Lemma \ref{lem:supermartingale}). For the enclosure $\chi_1$, for instance, the corresponding process $Y_n=\tr(\chi_1 \rho_n)$ should help us to select the trajectories absorbed in some sense in $\chi_\alpha$.
In Figure \ref{fig:traj} we trace the trajectories of $(Y_n)_n$ along $800$ steps, which show how $Y_\infty$ is a Bernoulli random variable with parameter $\mathbb{E}_\rho[\tr(A(\chi_1)\rho_0)]$; hence in this case $\pp^1_\rho(\cdot)$ (defined as in relation (\ref{eq:density2})) is equal to $ \pp_\rho( \cdot |B)$ where $B=\{Y_\infty=1\}=\{\lim_{n\rightarrow + \infty} \lVert p_{\chi_1} \rho_n p_{\chi_1}-\rho_n \rVert=0 \}$ and it represents the probability obtained conditioning $\pp_\rho$ to the event of ``being absorbed in $\chi_1$''.

\begin{figure}[!htbp]
        \centering
        \includegraphics[width=0.8\textwidth]{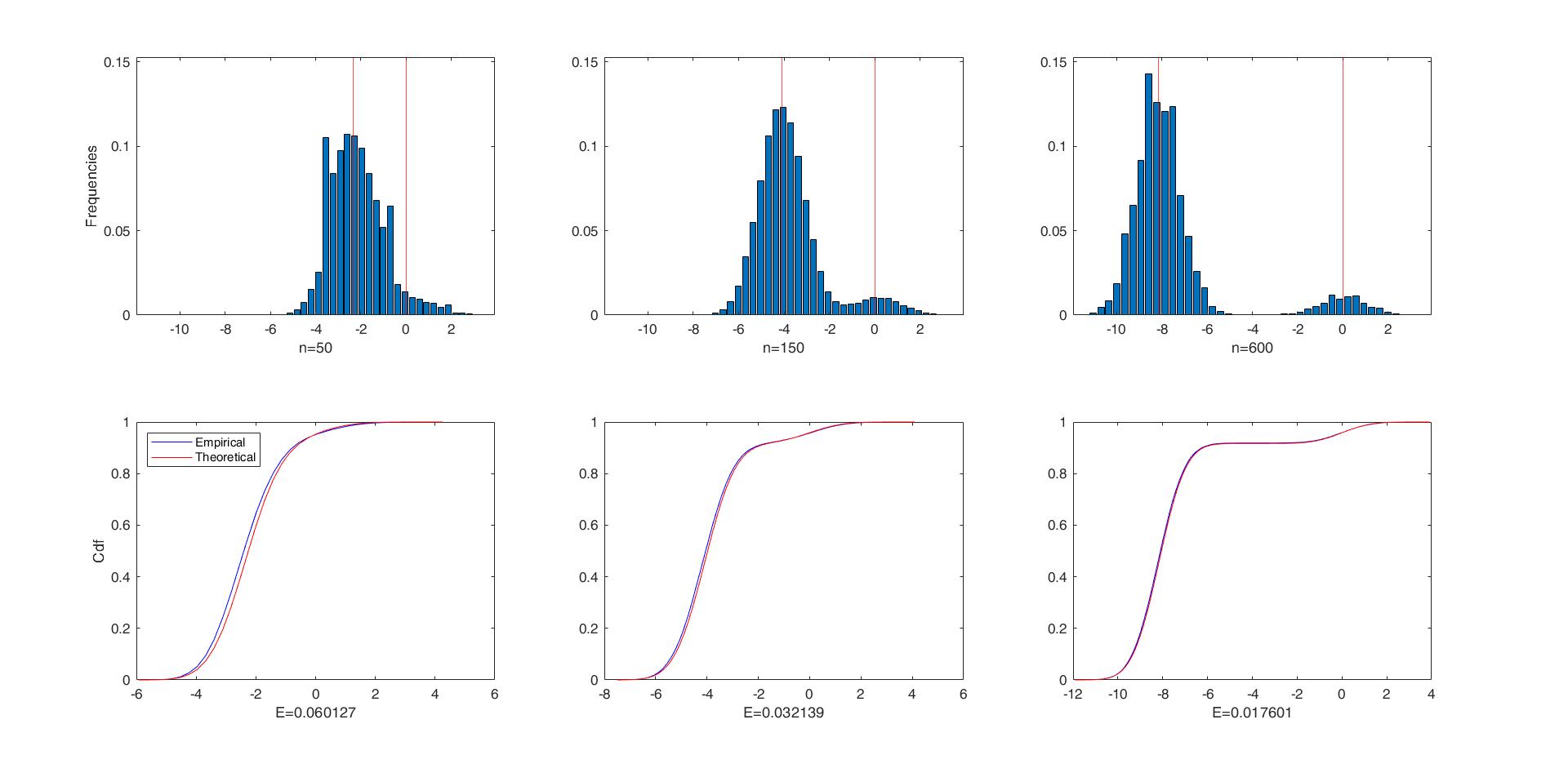} 
\caption{$\rho_0=\frac{1}{8} \ket{e_0}\bra{e_0}+\frac{7}{8} \ket{e_3}\bra{e_3}$ and $p_3=\frac{1}{6}$.} \label{fig:3}
\end{figure}

\begin{figure}[!htbp]
\centering
 \begin{subfigure}[b]{0.8\textwidth}
      \includegraphics[width=0.7\textwidth]{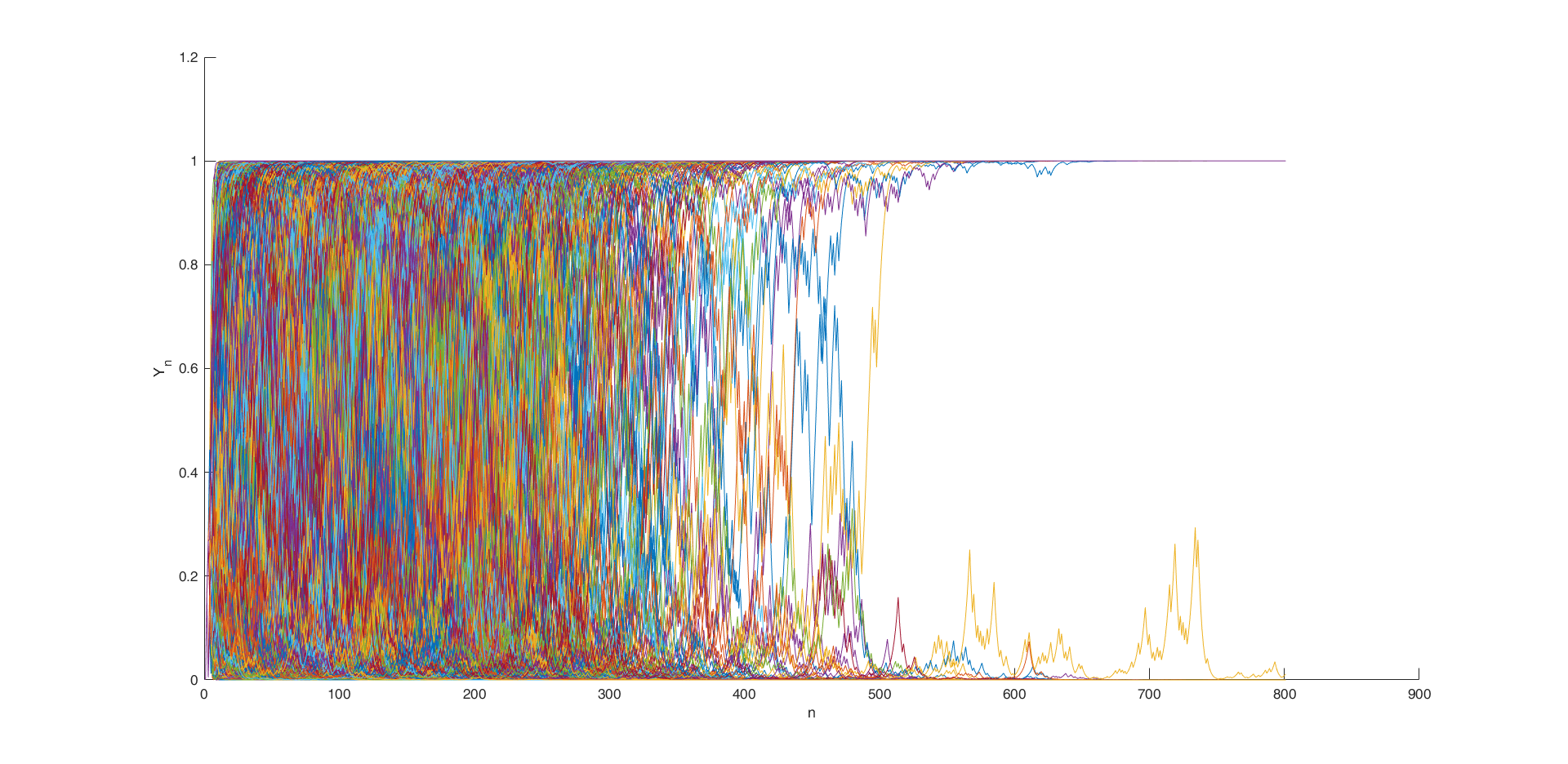} 
      \caption{The graph represents $N=10^4$ trajectories of $Y_n^1$ along $800$ steps ($\rho_0=\ketbra{e_0}{e_0}$ and $p_3=\frac{1}{6}$).}  \label{fig:traj}
\end{subfigure}
\begin{subfigure}[b]{0.8\textwidth}
      \includegraphics[width=0.8\textwidth]{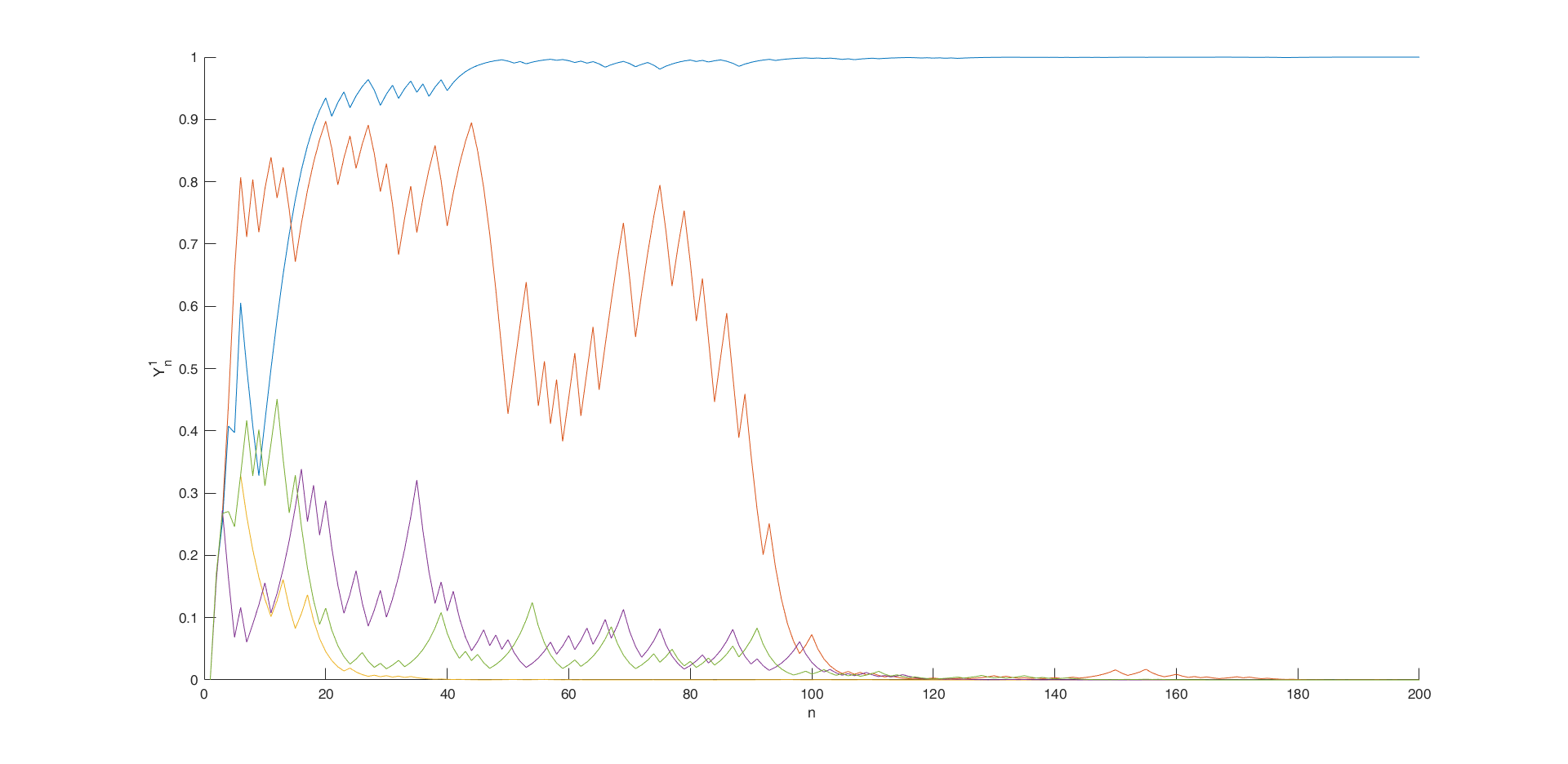} 
      \caption{The graph represents $N=5$ among the previous trajectories.}  \label{fig:traj}
\end{subfigure}
\caption{The behavior of $Y_{n}^{1}$.}
\end{figure}
Note that the frequence of trajectories such that $Y_{800} >0.99$ is equal to $0.3388$, and the frequence of trajectories for which $Y_{800} <0.01$ is $0.6612$. This is in agreement with $a_1(\rho)=\mathbb{E}_\rho[\tr(A(\chi_1)\rho_0)]=\frac{1}{3}$.

{\noindent\bf Acknowledgements.} The authors acknowledge the support of the INDAM GNAMPA project 2020 ``Evoluzioni markoviane quantistiche'' and of
the Italian Ministry of Education, University and Research (MIUR) for the
FFABR 2017 program and for the Dipartimenti di Eccellenza Program (2018–
2022)—Dept. of Mathematics ``F. Casorati'', University of Pavia.
\bibliographystyle{abbrv}
\bibliography{biblio}

\end{document}